\documentclass[10pt]{amsart}

\usepackage{amsmath}
\usepackage{float}
\usepackage{booktabs}
\usepackage[dvipsnames]{xcolor}

\usepackage{ wasysym }

\usepackage{blkarray}

\usepackage{array}
\newcolumntype{x}[1]{>{\centering\let\newline\\\arraybackslash\hspace{0pt}}p{#1}}

\usepackage{amscd,amsmath,amssymb,amsfonts,amsthm,ascmac}
\usepackage{enumerate,mathrsfs,stmaryrd,latexsym, comment, mathtools, mathdots} 

\usepackage[mathscr]{euscript}

\usepackage{caption}
\usepackage{subcaption}
\usepackage{hyperref}

\usepackage[all]{xy}
\usepackage{url}
\usepackage{ytableau}

\usepackage{tikz-cd}
\usepackage{tikz}
\usetikzlibrary{snakes, 
	3d, matrix,decorations.pathreplacing,calc,decorations.pathmorphing, patterns, fit,backgrounds}
\usetikzlibrary {positioning}
\usepackage{tikz-3dplot}

\numberwithin{equation}{section}
\allowdisplaybreaks[1]

\makeatletter
\newcommand{\leqnomode}{\tagsleft@true\let\veqno\@@leqno}
\newcommand{\reqnomode}{\tagsleft@false\let\veqno\@@eqno}
\makeatother

\newcommand{\defi}[1]{{\textit{#1}}}

\newcommand{\C}{{\mathbb{C}}}
\renewcommand{\P}{{\mathcal{P}}}

\renewcommand{\Pi}{P_{\mathbf{i}}}
\newcommand{\PGC}{\P_{\mathrm{GC}}}
\newcommand{\RGC}{\mathscr{R}_{\mathrm{GC}}}
\newcommand{\ffi}{f_{\Pi}}

\newcommand{\Pj}{P_{\mathbf{j}}}
\newcommand{\R}{{\mathbb{R}}}

\newcommand{\Z}{{\mathbb{Z}}}

\newcommand{\barn}{{\bar{n}}}
\newcommand{\WD}{\mathcal{WD}}

\newcommand{\ad}{\delta}

\newcommand{\Rn}[1]{{\mathscr{R}(w_0^{(#1)})}}

\newcommand{\node}{\mathrm{node}}

\DeclareMathOperator{\ind}{ind}
\DeclareMathOperator{\SL}{SL}

\DeclareMathOperator{\GC}{GC}

\newcommand{\A}{\mathsf{A}}
\newcommand{\D}{\mathsf{D}}
\newcommand{\cn}{\mathsf{c}}
\newcommand{\gc}{\mathsf{gc}}
\newcommand{\an}{\mathsf{a}}
\newcommand{\dn}{\mathsf{d}}

\newcommand{\RR}{\mathscr{R}}

\newtheorem{theorem}{Theorem}[section]
\newtheorem{lemma}[theorem]{Lemma}
\newtheorem{proposition}[theorem]{Proposition}
\newtheorem{corollary}[theorem]{Corollary}

\theoremstyle{definition}
\newtheorem{Question}[theorem]{Question}
\newtheorem{example}[theorem]{Example}
\newtheorem{definition}[theorem]{Definition}
\newtheorem{remark}[theorem]{Remark}

\begin{document}
	
\title[Enumeration of Gelfand--Cetlin type reduced words]{Enumeration of Gelfand--Cetlin type \\ reduced words}
\author{Yunhyung Cho}
\address{Department of Mathematics Education, Sungkyunkwan University, Seoul 03063, Republic of Korea}
\email{yunhyung@skku.edu}

\author{Jang Soo Kim}
\address{Department of Mathematics, Sungkyunkwan University, Suwon 16419, Republic of Korea}
\email{jangsookim@skku.edu}

\author{Eunjeong Lee}
\address{Center for Geometry and Physics, Institute for Basic Science (IBS), Pohang 37673, Republic of Korea}
\email{eunjeong.lee@ibs.re.kr}

\keywords{Commutation classes, Gelfand--Cetlin polytopes, posets, reduced words, reduced decompositions, Young tableaux of shifted shape}
\subjclass[2010]{Primary: 05Axx, 06A07, Secondary: 14M15, 52B20} 

\thanks{
	The first author was supported by the National Research Foundation of Korea(NRF) grant funded by the Korea government(MSIP; Ministry of Science, ICT and Future Planning) (NRF-2020R1C1C1A01010972). 
  The second was supported by NRF grants \#2019R1F1A1059081 and \#2016R1A5A1008055.
	The third author was supported by IBS-R003-D1.}
\begin{abstract}
The combinatorics of reduced words and commutation classes plays an important
  role in geometric representation theory. A string polytope is a lattice
  polytope associated to each reduced word of the longest element $w_0$ in the
  symmetric group which encodes the character of a certain irreducible
  representation of a Lie group of type $A$. In this paper, we provide a
  recursive formula for the number of reduced words of $w_0$ such that the
  corresponding string polytopes are combinatorially equivalent to a Gelfand--Cetlin
  polytope. The recursive formula involves the number of standard Young
  tableaux of shifted shape. We also show that each commutation class is
  completely determined by a list of quantities called indices.
\end{abstract}
\maketitle
\setcounter{tocdepth}{1} 
\date{\today}

\section{Introduction}

A {\em string polytope}, introduced by Littelmann~\cite{Li}, is a convex
polytope $\Delta_{\mathbf{i}}(\lambda)$ determined by two data: a reduced word
$\mathbf{i}$ of the longest element in the Weyl group of a reductive algebraic
group and a dominant weight $\lambda$. Its lattice points parametrize the dual
canonical basis elements of the irreducible representation with highest weight $\lambda$
so that it can be regarded as a non-abelian generalization of a Newton polytope
in toric geometry. The importance of string polytopes has been raised for the
study of mirror symmetry of flag varieties (see, for example,~\cite{BCKV}). We
refer the reader to~\cite{Lus90}, \cite{Kash90}, \cite{Li}, \cite{GlPo00},
and~\cite{BeZe01} for various descriptions of string polytopes.

One of the most famous examples of string polytopes is the Gelfand--Cetlin
polytope. Similarly to string polytopes, Gelfand--Cetlin polytopes have been
used to describe the irreducible representation of $\SL_{n+1}(\C)$ with highest
weight~$\lambda$. We recall the definition of Gelfand--Cetlin polytopes
from~\cite{GC1950} and~\cite{GS83}. Let $\lambda =
(\lambda_1,\dots,\lambda_n)$ be a sequence of nonnegative integers. For each~$\lambda$, define the
\defi{Gelfand--Cetlin polytope} $\GC(\lambda)$ to be the closed convex polytope
in $\R^{\barn}$ consisting of the points $(x_{k,j})_{1\le j\le k\le n}$ satisfying
the inequalities
\[
x_{k+1,j} \geq x_{k,j} \geq x_{k+1,j+1},\quad 1 \leq j \leq k \leq n,
\]
where $\barn = n(n+1)/2$, $x_{n+1,j} =\lambda_j+\cdots+\lambda_n$ for $1\le
j\le n$, and $x_{n+1,n+1} = 0$. Note that $\GC(\lambda)$ has the maximum
  dimension if each $\lambda_i$ is positive, i.e., $\lambda$ is regular. In this case, we say that
  $\GC(\lambda)$ is a full dimensional Gelfand--Cetlin polytope of rank~$n$.
 In~\cite[Corollary~5 in Section~5]{Li} it is shown that the
Gelfand--Cetlin polytope $\GC(\lambda)$ is an example of a string polytope of
$\SL_{n+1}(\C)$. More precisely, we have 
\[
\GC(\lambda) \simeq \Delta_{(1,2,1,3,2,1,\dots,n,n-1,\dots,1)}(\lambda),
\]
where $\simeq$ means the unimodular equivalence\footnote{Given integral
  polytopes $P \subset \R^d$ and $Q \subset \R^d$, we say that $P$ and $Q$ are
  \textit{unimodularly equivalent} if there exist a matrix $U \in \textup{M}_{d
    \times d}(\Z)$ satisfying $\det U = \pm 1$ and an integral vector $\mathbf v
  \in \Z^d$ such that $Q = f_U(P) + \mathbf v$. Here, $f_U$ is the linear
  transformation defined by $U$.}. In particular, these two
  polytopes are combinatorially equivalent. We refer the reader
to~\cite{ACK18} and references therein for more information on the
combinatorics of Gelfand--Cetlin polytopes.

The string polytope $\Delta_{\mathbf{i}}(\lambda)$ has the maximum dimension if
and only if the weight~$\lambda$ is regular.  Once the weight is assumed to be regular, the combinatorial type of the
string polytope is independent of the choice of the weight. In this paper, we
consider string polytopes $\Delta_{\mathbf{i}}(\lambda)$ of type~$A$ for a fixed
regular dominant weight $\lambda$ so that each
$\Delta_{\mathbf{i}}(\lambda)$ is determined by the reduced word $\mathbf{i}$.

The motivation of this paper is to enumerate the string polytopes
$\Delta_{\mathbf{i}}(\lambda)$ of type~$A$ which are combinatorially equivalent to
Gelfand--Cetlin polytopes. To this end, we study the reduced words $\mathbf{i}$
which give rise to such string polytopes $\Delta_{\mathbf{i}}(\lambda)$. To
state our results, we introduce some terminologies.

Let $\frak{S}_{n+1}$ be the symmetric group (i.e., the Weyl group of
$\mathrm{SL}_{n+1}(\C)$) of degree~$n+1$ and denote by $s_i := (i, i+1)$ the
simple transposition which swaps $i$ and~$i+1$ and fixes all other elements of
$[n+1] = \{1, \dots, n+1\}$. The set $\{s_1, \dots, s_n\}$ of simple
transpositions generates $\frak{S}_{n+1}$, hence every element $w \in
\frak{S}_{n+1}$ can be written in the following form:
\[
  w = s_{i_1} \cdots s_{i_r}, \quad \quad i_1, \dots, i_r \in [n].
\]
In this case, the sequence ${\bf i} := (i_1, \dots, i_r)$ is called a {\em word}
of $w$. The {\em length} $\ell(w)$ of~$w$ is defined to be the smallest integer
$r$ for which $(i_1,\dots,i_r)$ is a word of $w$. A word $(i_1, \dots, i_r)$ of
$w$ is \emph{reduced} if $r=\ell(w)$. We denote by $\RR(w)$ the set of reduced
words of $w$. There is a unique element $w_0^{(n+1)}$, called the {\em longest
element}, in $\mathfrak{S}_{n+1}$ such that $\ell(w)\le\ell(w_0^{(n+1)})$ for all
$w\in\mathfrak{S}_{n+1}$.

For $i,j \in [n]$ satisfying $|i-j| > 1$, we have $s_i s_j = s_j s_i$. This
induces an operation on the set $\RR(w)$ defined by $(\dots, i, j, \dots)
\mapsto (\dots, j,i, \dots)$, which is called a {\em commutation} (or a {\em
  2-move}). Define an equivalence relation $\sim$ on $\RR(w)$ by
\[
{\bf i} \sim {\bf i}' \quad \Leftrightarrow \quad \text{${\bf i}$ is obtained
  from ${\bf i}'$ by a sequence of commutations}.
\]
An element in $[\RR(w)] := \RR(w) / \sim$ is called a {\em commutation class} for $w$.	
For a recent account of the study of commutation classes, we refer the reader to
\cite{Bedard99, STWW17, FMPTE19, GMS} and references therein.

One important fact about commutation classes for our purpose is that two string
polytopes $\Delta_{\mathbf{i}}(\lambda)$ and $\Delta_{\mathbf{i}'}(\lambda)$ are
combinatorially equivalent if (but not necessarily only if) ${\bf i}$ and
${\bf i}'$ are in the same commutation class (see~\cite[Lemma~3.1]{CKLP}).
Accordingly, studying the elements in $[\Rn{n+1}]$ is closely related to the
classification problem of the combinatorial types of string polytopes.

We say that $\mathbf{i}\in\Rn{n+1}$ is a {\em Gelfand--Cetlin type} reduced word
if the corresponding string polytope $\Delta_{\mathbf{i}}(\lambda)$ is
combinatorially equivalent to a full dimensional Gelfand--Cetlin
polytope of rank $n$. Let $\gc(n)$ be the number of Gelfand--Cetlin type reduced words
in~$\Rn{n+1}$. By the definition of $\gc(n)$, it also counts the number of
string polytopes $\Delta_{\mathbf{i}}(\lambda)$ with~$\mathbf{i}\in\Rn{n+1}$
that are combinatorially equivalent to a full dimensional Gelfand--Cetlin
polytope of rank $n$.

The first main result in this paper is the following recurrence relation for
$\gc(n)$.

\begin{theorem}[Theorem~\ref{thm_SYT}]\label{thm_main_1}
	The number $\gc(n)$ of Gelfand--Cetlin type reduced words in~$\Rn{n+1}$
  satisfies
	\[
	\gc(n) = \sum_{k=1}^n g^{(n,n-1,\dots,n-k+1)} \gc(n-k),
	\]
	where for $\mu=(\mu_1,\dots,\mu_t)$,
\[
g^{\mu} = \frac{|\mu|!}{\mu_1 ! \mu_2 ! \cdots \mu_t !}
\prod_{i < j} \frac{\mu_i - \mu_j}{\mu_i + \mu_j},
\]  
which is the number of standard Young tableaux of shifted shape $\mu$.
\end{theorem}

As a consequence of the proof of the above theorem, we obtain that the number of
commutation classes consisting of Gelfand--Cetlin type reduced words in
$\Rn{n+1}$ is $2^{n-1}$ (see Corollary~\ref{cor_main}). This result was also
proved in a recent paper~\cite{GMS} using a different method.

We note that the number of string polytopes $\Delta_{\mathbf i}(\lambda)$ (for $\mathbf i \in \RR(w_0^{(n+1)})$) which are \textit{unimodularly} equivalent to the Gelfand--Cetlin polytope $\GC(\lambda)$ is the same as the number $\gc(n)$. Accordingly, the above theorem also enumerates the number of string polytopes $\Delta_{\mathbf i}(\lambda)$ which are unimodularly equivalent to the Gelfand--Cetlin polytope $\GC(\lambda)$ (see Corollary~\ref{cor_main2}).

A crucial object in the proof of Theorem~\ref{thm_main_1} is a quantity called
$\delta$-index. For a sequence $\delta\in\{\A,\D\}^{n-1}$ of two letters $\A$
and $\D$, the \emph{$\delta$-index} $\ind_\delta(\mathbf{i})$ of a reduced word
$\mathbf{i}\in\Rn{n+1}$ is an element in $\Z^{n-1}$ which measures how far a
given word is from the standard reduced word
\[ (1,2,1,3,2,1, \dots, n,n-1, \dots,1)
\] of $w_0^{(n+1)}$. See Section~\ref{secContractionsExtensionsAndIndices} for
the precise definition. 

Recently, the first and the third authors together with Kim and
Park~\cite[Theorem~A]{CKLP} classified all Gelfand--Cetlin type reduced words in
$\Rn{n+1}$ in terms of $\delta$-indices. More precisely, they showed that
$\mathbf{i}\in\Rn{n+1}$ is a Gelfand--Cetlin reduced word if and only if there
is a sequence $\delta\in\{\A,\D\}^{n-1}$ such that
$\ind_\delta(\mathbf{i})=(0,\dots,0)\in\Z^{n-1}$.

It turns out that for $\mathbf{i}, \mathbf{j}\in\Rn{n+1}$, if
$\ind_\delta(\mathbf{i}) = \ind_\delta(\mathbf{j})=(0,\dots,0)\in\Z^{n-1}$ for
some $\delta\in\{\A,\D\}^{n-1}$, then $\mathbf{i}\sim\mathbf{j}$. However, the
condition $\ind_{\delta}(\mathbf{i}) = \ind_{\delta}(\mathbf{j})$ for some
$\delta\in\{\A,\D\}^{n-1}$ does not always imply $\mathbf{i}\sim\mathbf{j}$ (see
Example~\ref{example_total_index}).

The second main result in this paper is the following theorem, which shows that
the $\delta$-indices $\ind_\delta(\mathbf{i})$ for all
$\delta\in\{\A,\D\}^{n-1}$ completely determine the commutation class of
$\mathbf{i}$.

\begin{theorem}[Theorem~\ref{thm_main_injection}]\label{thm_main_2}
  Let $\mathbf{i}, \mathbf{j}\in\Rn{n+1}$. If $\ind_\delta(\mathbf{i}) =
  \ind_\delta(\mathbf{j})$ for all $\delta\in\{\A,\D\}^{n-1}$, then
  $\mathbf{i}\sim\mathbf{j}$.
\end{theorem}

In order to prove our main results, we consider word posets, which are similar to
wiring diagrams. A word poset is a poset $P$ together with a function $f_P:P\to
\Z_{>0}$. Each commutation class in $[\Rn{n+1}]$ corresponds to a word poset and
the cardinality of the commutation class is equal to the number of linear
extensions of the corresponding word poset.

This paper is organized as follows. In
Section~\ref{secPosetsOfReducedDecompositions}, we give basic definitions. In
Section~\ref{secContractionsExtensionsAndIndices}, we recall the operations on
$\Rn{n+1}$ called contractions and extensions. Moreover, we provide the
definition of indices and prove several properties of them. In
Section~\ref{secGelfandCetlinTypeCommutationClasses}, we study Gelfand--Cetlin
type reduced words and provide a proof of Theorem~\ref{thm_main_1}. In
Section~\ref{secInjectivityOfTheTotalIndexMap}, we provide a proof of
Theorem~\ref{thm_main_2}.

\section{Basic definitions}
\label{secPosetsOfReducedDecompositions}

In this section, we give basic definitions and some of their properties which
will be used throughout this paper.

\subsection{Commutation classes}

Let $w$ be an element in $\frak{S}_{n+1}$ and denote by $\RR(w)$ the set of reduced words of~$w$, i.e., 
\[
\RR(w) = \{ (i_1,\dots,i_{\ell}) \in [n]^{\ell} \mid s_{i_1} s_{i_2} \cdots s_{i_{\ell}} = w\},
\]
where  $\ell=\ell(w)$ is the length of $w$.
It is well known that 
\[
\ell(w) = \# \{1 \leq i < j \leq n \mid w(i) > w(j) \}.
\]
We denote by $w_0^{(n+1)}$ the longest element 
$n+1 \ n \ \dots
  \ 1$ of $\frak{S}_{n+1}$ (using the one-line notation). The length $\barn$ of the
longest element $w_0^{(n+1)}$ is given by
\begin{equation}\label{eq_def_of_nbar}
\barn =  \ell(w_0^{(n+1)})= \frac{n (n+1)}{2} \quad \text{ for } n \in \Z_{>0}:=\{1,2,\dots\}.
\end{equation}

Recall that for a given ${\bf i} \in \RR(w)$, one can produce new reduced
words using {\em braid moves}. There are two types of braid moves as follows:
\begin{itemize}
\item A \emph{2-move} replaces two consecutive elements $i,j$ in ${\bf i}$ by
  $j,i$ for some integers $i$ and $j$ with $|i-j| > 1$.
\item A \emph{3-move} replaces three consecutive elements $i,j,i$ in ${\bf i}$ by
  $j,i,j$ for some integers $i$ and $j$ with $|i-j| = 1$.
\end{itemize}
Note that braid moves do not change the product of the simple transpositions
for $\mathbf{i}$ since $s_i s_j = s_j s_i$ for $|i-j|>1$ and $s_i s_{i+1} s_i =
s_{i+1} s_i s_{i+1}$. According to Tits' Theorem \cite{Ti}, any
two reduced words in $\RR(w)$ are connected by a sequence of braid moves. Braid
moves and Tits' theorem can be generalized to other Coxeter systems
(see~\cite[\S3.3]{BB05Combinatorics}).

Define an equivalence relation `$\sim$' on $\RR(w)$ by
\[
	{\bf i} \sim {\bf i}' \quad \Leftrightarrow \quad \text{${\bf i}$ is obtained from ${\bf i}'$  by a sequence of 2-moves}.
\]
We denote by $[\RR(w)] \coloneqq \RR(w)/\sim$ the set of equivalence classes and call an element~$[{\bf i}] \in [\RR(w)]$ a {\em commutation class}.

\begin{remark}\label{rmk_cn}
	There is no known exact formula for the number $\cn(n+1)$ of commutation
	classes of $w_0^{(n+1)}$. Some upper and lower bounds for $\cn(n)$ were
	provided by Knuth \cite[Section~9]{Knu}. Felsner and Valtr \cite[Theorem~2 and
	Proposition~1]{FV11} found the following upper and lower bounds for $\cn(n+1)$
	improving Knuth's results: for a sufficiently large~$n$,
	\[
	2^{0.1887n^2} \leq \cn(n) \leq 2^{0.6571n^2}.
	\] 
	The first few terms of $\cn(n)$ are $1, 1, 2, 8, 62, 908, 24698, 1232944$, see
	\href{https://oeis.org/A006245}{A006245} in~\cite{OEIS}.
\end{remark}

\subsection{Wiring diagrams}\label{rmk_P_i_and_wiring_diagram}

There are several combinatorial models for the commutation classes of the
longest element of $\mathfrak{S}_{n+1}$, see~\cite{DES} and references therein.
We recall a well known combinatorial model, called a wiring diagram
(cf.~\cite{GoodmanPollack93}).

\begin{definition}
	Let  $\mathbf i = (i_1,i_2,\dots,i_{\ell})$ be a word  of $w \in
  \mathfrak{S}_{n+1}$.
  \begin{enumerate}
 \item The \emph{wiring diagram} $G(\mathbf i)$ of $\mathbf{i}$
  is a collection of line segments defined as follows.
  \begin{itemize}
  \item For $1\le i\le n$ and  $1\le j\le n+1$, define $G_i$ to be the collection of line segments
    $(A_j,B_{s_i(j)})$ connecting $A_j$ and $B_{s_i(j)}$,
    where $A_j=(j,1)$ and $B_j=(j,0)$ are points in $\mathbb{R}^2$. The
    intersection of the segments $(A_i,B_{i+1})$ and $(A_{i+1},B_i)$ is called a
    \emph{crossing} in column $i$. See Figure~\ref{fig:G_r}.
  \item Define $G(\mathbf{i})$ to be the configuration obtained by arranging
    $G_{i_1},G_{i_2},\dots,G_{i_\ell}$ vertically in this order. More precisely,
\[
G(\mathbf{i})=\rho^{\ell-1}(G_{i_1})\cup \rho^{\ell-2}(G_{i_2})\cup \dots\cup
    \rho^{0}(G_{i_\ell}), 
\]
where $\rho$ is the translation by $(0,1)$. See Figure~\ref{fig:WD}.
\end{itemize}

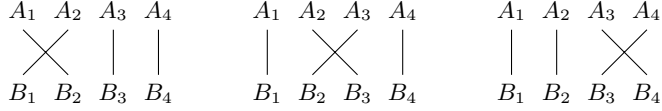
\begin{figure}
  \centering
 \begin{tikzpicture}[scale = 0.6]
		\tikzset{every node/.style = {font = \small}}
  \coordinate (b1) at (0,0);
 \coordinate (b2) at (1,0);
 \coordinate (b3) at (2,0);
 \coordinate (b4) at (3,0);
 
 \coordinate (a1) at (0,1);
 \coordinate (a2) at (1,1);
 \coordinate (a3) at (2,1);
 \coordinate (a4) at (3,1);
 
 \draw (b1)--(a2)
 (b2)--(a1)
 (b3)--(a3)
 (b4)--(a4);
 
  \node[below] at (b1)  {$B_1$};
 \node[below] at (b2) {$B_2$};
 \node[below] at (b3) {$B_3$};
 \node[below] at (b4) {$B_4$};
 \node[above] at (a1) {$A_1$};
 \node[above] at (a2) {$A_2$};
 \node[above] at (a3) {$A_3$};
 \node[above] at (a4) {$A_4$};
 \end{tikzpicture} \qquad 
 \begin{tikzpicture}[scale = 0.6]
 \tikzset{every node/.style = {font = \small}}
 \coordinate (b1) at (0,0);
 \coordinate (b2) at (1,0);
 \coordinate (b3) at (2,0);
 \coordinate (b4) at (3,0);
 
 \coordinate (a1) at (0,1);
 \coordinate (a2) at (1,1);
 \coordinate (a3) at (2,1);
 \coordinate (a4) at (3,1);
 
 \draw (b1)--(a1)
 (b2)--(a3)
 (b3)--(a2)
 (b4)--(a4);
 
 \node[below] at (b1)  {$B_1$};
 \node[below] at (b2) {$B_2$};
 \node[below] at (b3) {$B_3$};
 \node[below] at (b4) {$B_4$};
 \node[above] at (a1) {$A_1$};
 \node[above] at (a2) {$A_2$};
 \node[above] at (a3) {$A_3$};
 \node[above] at (a4) {$A_4$};
 \end{tikzpicture} \qquad
 \begin{tikzpicture}[scale = 0.6]
 \tikzset{every node/.style = {font = \small}}
 \coordinate (b1) at (0,0);
 \coordinate (b2) at (1,0);
 \coordinate (b3) at (2,0);
 \coordinate (b4) at (3,0);
 
 \coordinate (a1) at (0,1);
 \coordinate (a2) at (1,1);
 \coordinate (a3) at (2,1);
 \coordinate (a4) at (3,1);
 
 \draw (b1)--(a1)
 (b2)--(a2)
 (b3)--(a4)
 (b4)--(a3);
 
 \node[below] at (b1)  {$B_1$};
 \node[below] at (b2) {$B_2$};
 \node[below] at (b3) {$B_3$};
 \node[below] at (b4) {$B_4$};
 \node[above] at (a1) {$A_1$};
 \node[above] at (a2) {$A_2$};
 \node[above] at (a3) {$A_3$};
 \node[above] at (a4) {$A_4$};
 \end{tikzpicture}
  \caption{The configurations $G_1,G_2$, and $G_3$ (from left to right) for $n=4$.}
  \label{fig:G_r}
\end{figure}

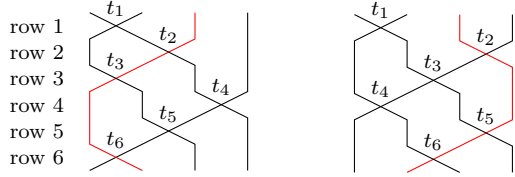
\begin{figure}
  \centering
		\begin{tikzpicture}[scale = 0.7]
		\tikzset{every node/.style = {font = \footnotesize}}
		\tikzset{red line/.style = {line width=0.5ex, red, semitransparent}}
		\tikzset{blue line/.style = {line width=0.5ex, blue, semitransparent}}
		
		\draw (0,0)--(3,1.5)--(3,3);
		\draw[red] (1,0)--(1,0)--(0,0.5)--(0,1.5)--(2,2.5)--(2,3);
		\draw (2,0)--(2,0.5)--(1,1)--(1,1.5)--(0,2)--(0,2.5)--(1,3)--(1,3);
		\draw (3,0)--(3,1)--(2,1.5)--(2,2)--(0,3)--(0,3);
		
		\node[above] at (0.5,2.75) {$t_1$};
		\node[above] at (1.5, 2.25) {$t_2$};
		\node[above] at (0.5,1.75) {$t_3$};
		\node[above] at (2.5, 1.25) {$t_4$};
		\node[above] at (1.5,0.75) {$t_5$};
		\node[above] at (0.5,0.25) {$t_6$};
		
		\node at (-1, 2.75) {row $1$};
		\node at (-1, 2.25) {row $2$};
		\node at (-1, 1.75) {row $3$};
		\node at (-1, 1.25) {row $4$};
		\node at (-1, 0.75) {row $5$};
		\node at (-1,0.25) {row $6$};
		\end{tikzpicture}\qquad\qquad
		\begin{tikzpicture}[scale = 0.7]
		\tikzset{every node/.style = {font = \footnotesize}}
		\tikzset{red line/.style = {line width=0.5ex, red, semitransparent}}
		\tikzset{blue line/.style = {line width=0.5ex, blue, semitransparent}}
		
		
		\draw (0,0)--(0,1)--(3,2.5)--(3,3);
		\draw[red] (1,0)--(1,0)--(3,1)--(3,2)--(2,2.5)--(2,3);
		\draw (2,0)--(2,0)--(1,0.5)--(1,1)--(0,1.5)--(0,2.5)--(1,3)--(1,3);
		\draw (3,0)--(3,0)--(3,0.5)--(2,1)--(2,1.5)--(1,2)--(1,2.5)--(0,3)--(0,3);

		\node[above] at (0.5, 2.75) {$t_1$};
		\node[above] at (2.5, 2.25) {$t_2$};
		\node[above] at (1.5, 1.75) {$t_3$};
		\node[above] at (0.5, 1.25) {$t_4$};
		\node[above] at (2.5, 0.75) {$t_5$};
		\node[above] at (1.5, 0.25) {$t_6$};

		\end{tikzpicture}
 \caption{The wiring diagrams $G(\mathbf{i})$ (left) and $G(\mathbf{j})$ (right)
   for $\mathbf{i}=(1,2,1,3,2,1)$ and $\mathbf{j}=(1,3,2,1,3,2)$ in
   $\RR(w_0^{(4)})$. The $3$rd wire in each wiring diagram is colored red. The
   crossing in row $j$ is labeled by $t_j$. The crossing $t_4$ in
   $G(\mathbf{i})$ (respectively, $G(\mathbf{j})$) is in column $3$
   (respectively column $1$).}
  \label{fig:WD}
\end{figure}

\item The \emph{$j$th row} of $G(\mathbf{i})$ is $\rho^{\ell-j}(G_{i_j})$. The points
$(j,\ell)$ and $(j,0)$ are called the \emph{$j$th starting point} and the
\emph{$j$th ending point} of $G(\mathbf{i})$, respectively. A \emph{wire} of
$G(\mathbf{i})$ is a path from a starting point to an ending point of
$G(\mathbf{i})$ obtained by taking the union of $\ell$ segments one from each
row. If a wire starts at the $j$th starting point, it is called the
\emph{$j$th wire} of $G(\mathbf{i})$.

\item We denote by $\WD(w)$ the set of wiring diagrams $G(\mathbf{i})$ for all
reduced words~$\mathbf{i}$ of~$w$.
\end{enumerate}
\end{definition}

By the definition of $G(\mathbf{i})$, it is clear that the $j$th wire is from
the $j$th starting point to the $w(j)$th ending point. One can reconstruct
$\mathbf{i}$ from $G(\mathbf{i})$ because the unique crossing in row $j$ of
$G(\mathbf{i})$ is in column $i_j$. This gives a bijection between $\RR(w)$ and
$\WD(w)$. We define the equivalence relation `$\sim$' on $\WD(w)$ by
$G(\mathbf{i})\sim G(\mathbf{j})$ if and only if $\mathbf{i}\sim\mathbf{j}$.
Equivalently, we have $G(\mathbf{i})\sim G(\mathbf{j})$ if and only if
$G(\mathbf{i})$ is obtained from $G(\mathbf{j})$ by a sequence of operations
exchanging two adjacent rows in which the crossings are not in adjacent columns.

Let $[\WD(w)]=\WD(w)/\sim$. Then we have an obvious bijection between $[\WD(w)]$
and $[\RR(w)]$ induced by the correspondence explained above.

\subsection{Word posets}

To study commutation classes, we associate a poset and a function on the poset to
each reduced word $\mathbf i \in \RR(w)$. 

\begin{definition}
  A \emph{word poset} is a poset $P$ together with a function $f_P:P\to \Z_{>0}$.
  Two word posets $P$ and $Q$ are \emph{isomorphic}, denoted by
  $P\sim Q$, if there is a poset isomorphism $\phi:P\to Q$ such that
  $f_P(x)= f_Q(\phi(x))$ for all $x\in P$.
\end{definition}

\begin{definition}
	Let $\mathbf i = (i_1,\dots,i_{\ell}) \in \RR(w)$. The \emph{word poset} of
  $\mathbf{i}$ is a poset $\Pi$ on~$[\ell]$ together with a function
  $\ffi:\Pi\to \Z_{>0}$ defined as follows:
  \begin{itemize}
  \item $\Pi$ is the transitive closure of the binary relation containing
    $(r,r)$ for $r\in[\ell]$ and $(j,k)$ for $j,k\in [\ell]$ with $|i_j - i_k| =
    1$ and $j < k$, and
  \item $\ffi(j) = i_j$ for all $j\in\Pi$.
  \end{itemize}
  Denote by $\P(w)$ the set of word posets $P$ such that $P\sim \Pi$ for some
 $\mathbf i  \in \RR(w)$. We also define $[\P(w)]=\P(w)/\sim$.
\end{definition}

We will see later in this subsection that word posets are closely related to
wiring diagrams.
Throughout this paper the following convention will be used when we draw the
Hasse diagram of a word poset.

\smallskip
\noindent\textbf{Convention.}
Let $P$ be a word poset. For $j\in P$, if $f_P(j)=i$, we say that $j$ \emph{is
  in column} $i$. When we draw the Hasse diagram of $P$ the element $j$ will be
placed in column $i$. For a subset $A$ of $P$, we say that $Q$ is the word poset
obtained from $P$ by \emph{shifting} $A$ to the left (respectively, right) by one
column if $P$ and $Q$ are the same as posets and $f_Q(x)=f_P(x)$ if $x\notin A$,
and $f_Q(x)=f_P(x)-1$ (respectively, $f_Q(x)=f_P(x)+1$) if $x\in A$.

\medskip

By definition, every covering relation in a word poset $P\in \P(w)$ occurs
between two adjacent columns. In other words, if $x\lessdot_P y$, then
$|f_P(x)-f_P(y)|=1$.

\begin{example}\label{example_poset}
	We illustrate some examples of the word posets $\Pi$ associated with
  some reduced words $\mathbf{i}=(i_1,\dots,i_\ell)$. We will arrange the
  elements of $\Pi$ so that $j$ is in column $\ffi(j)=i_j$.
	\begin{enumerate}
\item	Let $\mathbf i = (1,2,1,3,2,1)$. Then the Hasse diagram of the word poset
  $\Pi$ is shown below.
	
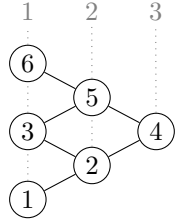
\begin{center}
	\begin{tikzpicture}[auto,node distance=0.1cm and 0.5cm, inner sep=2pt]
\tikzstyle{state}=[draw, circle]
	\node[state] (1) {$6$};
	\node[state] (2) [below right = of 1] {$5$};
	\node[state] (3) [below left = of 2] {$3$};
	\node[state] (4) [below right = of 2] {$4$};
	\node[state] (5) [below left = of 4] {$2$};
	\node[state] (6) [below left= of 5] {$1$};
	
	\draw (1)-- (2)
	(2)--(3)
	(2)--(4)
	(3)--(5)
	(4)--(5)
	(5)--(6);
	
	\draw[dotted, gray] (6)--(3)--(1)--(0,0.5) node[above] {$1$};
	\draw[dotted, gray] (5)--(2)--(0.85,0.5) node[above] {$2$};
	\draw[dotted, gray] (4)--(1.7, 0.5) node[above] {$3$};

	\end{tikzpicture}
	\end{center}
  Here the elements $1,3,6$ are in column $1$, the elements $2,5$ are in column~$2$, and the element $4$ is in column $3$.
  Note that if $\mathbf{i}' = (1,2,3,1,2,1)$, then $\mathbf{i}'\sim \bf i$ and
  the word poset $P_{\mathbf{i}'}$ shown below is isomorphic to $\Pi$.
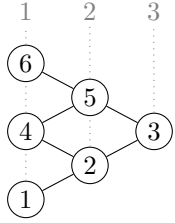
	\begin{center}
	\begin{tikzpicture}[auto,node distance=0.1cm and 0.5cm, inner sep=2pt]
\tikzstyle{state}=[draw, circle]
	\node[state] (1) {$6$};
	\node[state] (2) [below right = of 1] {$5$};
	\node[state] (3) [below left = of 2] {$4$};
	\node[state] (4) [below right = of 2] {$3$};
	\node[state] (5) [below left = of 4] {$2$};
	\node[state] (6) [below left= of 5] {$1$};
	
	\draw (1)-- (2)
	(2)--(3)
	(2)--(4)
	(3)--(5)
	(4)--(5)
	(5)--(6);
	
		\draw[dotted, gray] (6)--(3)--(1)--(0,0.5) node[above] {$1$};
	\draw[dotted, gray] (5)--(2)--(0.85,0.5) node[above] {$2$};
	\draw[dotted, gray] (4)--(1.7, 0.5) node[above] {$3$};
	\end{tikzpicture}
	\end{center}
\item	Let $\mathbf j = (1,3,2,1,3,2)$. Then the  Hasse diagram of the word  poset
  $\Pj$ is given as follows.  
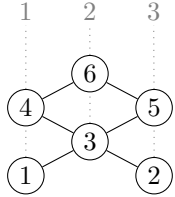
\begin{center}
	\begin{tikzpicture}[auto,node distance=0.1cm and 0.5cm, inner sep=2pt]
\tikzstyle{state}=[draw, circle]
\node[state] (1) {$1$};
\node[state] (3) [above right = of 1] {$3$};
\node[state] (2) [below right = of 3] {$2$};
\node[state] (4) [above left = of 3] {$4$};
\node[state] (5) [above right = of 3] {$5$};
\node[state] (6) [above right = of 4] {$6$};

\draw (1)--(3)
(3)--(2)
(3)--(4)
(3)--(5)
(4)--(6)
(6)--(5);

	\draw[dotted, gray] (1)--(4)--(0,2) node[above] {$1$};
\draw[dotted, gray] (3)--(6)--(0.85,2) node[above] {$2$};
\draw[dotted, gray] (2)--(5)--(1.7, 2) node[above] {$3$};

\end{tikzpicture}
\end{center}
	\end{enumerate}
\end{example}

A \emph{linear extension} of a poset $P$ is a permutation $p_1p_2\dots p_n$
of the elements in~$P$ such that $j<k$ whenever $p_j<_P p_k$.

\begin{proposition}\label{prop_poset_determines_commutations_class}{\cite[Theorem~1.1]{Samuel11}}
	Let $\mathbf i, \mathbf j \in \RR(w)$. 
	\begin{enumerate}
  \item There is a bijection between the linear extensions of the poset $\Pi$
    and the elements in the commutation class $[\bf i]$ given as follows. A
    linear extension $p_1p_2\dots p_n$ of $\Pi$ corresponds to the word
    $(i_{\ffi(p_1)},i_{\ffi(p_2)},\dots,i_{\ffi(p_\ell)})$ in $[\mathbf{i}]$.

  \item We have $\bf i\sim \bf j$ if and only if $\Pi\sim \Pj$.
	\end{enumerate}
\end{proposition}

We note that Proposition~\ref{prop_poset_determines_commutations_class}(1) has a
similar result on partial commutation monoids, see \cite[\S 5.1.2, Exercise
11]{knuth}, \cite[Exercise 3.123]{ec1}, or
\cite[Proposition~4.11]{DB-Kim-Reiner}.

By Proposition~\ref{prop_poset_determines_commutations_class}(2), we can identify
the commutation class $[\mathbf i]$ with $[\Pi]$.

\begin{proposition}\label{prop:i=P}
  The map $[\mathbf{i}]\mapsto [\Pi]$ is a bijection from $[\RR(w)]$ to
  $[\P(w)]$. 
\end{proposition}

There is also a direct and natural correspondence between $[\WD(w)]$ and
$[\P(w)]$, which we now explain. Let $[D]\in [\WD(w)]$. We define the corresponding word
poset class $[P]\in[\P(w)]$ as follows.
\begin{itemize}
\item The elements of the underlying set $P$ are the
crossings in $D$. 
\item For two distinct elements $a,b\in P$, we have $a<_Pb$ if there is a
  downward path from the crossing $a$ to the crossing $b$ in $D$. Here, a
  downward path means a path following wires (it is allowed to switch between
  the two wires at a crossing) in the direction that the $y$-coordinate
  decreases.
\item For $a\in P$, define $f_P(a)=c$ if the crossing $a$ is in column $c$ in
  $D$.
\end{itemize}

For example, the wiring diagrams $G(\mathbf{i})$ and $G(\mathbf{j})$ in
Figure~\ref{fig:WD} correspond to the word posets $\Pi$ and $\Pj$ in
Example~\ref{example_poset}.

\begin{proposition}\label{prop:wiring}
  Let $w\in \mathfrak{S}_{n+1}$. Then the map $[D]\mapsto [P]$ described above
  is a bijection from $[\WD(w)]$ to $[\P(w)]$.
\end{proposition}
\begin{proof}
  We first show that the map $[D]\mapsto [P]$ is well defined. Suppose that $D'$
  is obtained from $D$ by exchanging two adjacent rows in which the crossings
  are not in adjacent columns. It is easy to see that a downward path from a
  crossing $a$ to a crossing $b$ in $D$ remains a downward path from $a$ to $b$
  in $D'$. This shows that the images of $[D]$ and $[D']$ under this map are
  identical. Thus the map is well defined.

  Now we show that the map is a bijection. Since $[\WD(w)]=\{[G(\mathbf{i})]:
  \mathbf{i}\in \RR(w)\}$ and $[\P(w)]=\{[P_{\mathbf{i}}]: \mathbf{i}\in
  \RR(w)\}$ both in bijection with $[\RR(w)]$, it suffices to show that
  $[G(\mathbf{i})]\mapsto[P_{\mathbf{i}}]$. This is straightforward to check
  by the construction of the map. We omit the details.
\end{proof}

\section{Contractions, extensions, and indices}
\label{secContractionsExtensionsAndIndices}

In this section, we define $\A$-, $\D$-, and $\delta$-indices of a word poset,
and two operations on word posets, called contractions and extensions.
We also provide some results on these objects which will be used in later sections.

From now on, we concentrate on the reduced words of the longest element
$w_0^{(n+1)}$ in the symmetric group $\mathfrak{S}_{n+1}$ and the word posets in
$\P(w_0^{(n+1)})$. In order to define indices, we prepare the following two
lemmas.

\begin{lemma}[{cf.~\cite[\S 2.2]{GlPo00}}]\label{lemma_properties_of_longest_elt}
	Let $\mathbf i \in \RR(w_0^{(n+1)})$. 
The wiring diagram $G(\mathbf i)$  has the following properties.
\begin{enumerate}
  \item The $i$th wire travels from the $i$th starting point to the $(n+2-i)$th
    ending point for $1\le i\le n+1$.
	\item Two different wires meet exactly once.
  \item Every wire has exactly $n$ crossings.
  \item Let $c_1,\dots,c_n$ be the crossings that lie on the $1$st wire
    \textup{(}respectively, $(n+1)$st wire\textup{)} of $G(\mathbf{i})$ in this order, i.e., the
    row of $c_{i+1}$ is lower than the row of $c_i$. Then each $c_i$ is in
    column $i$ \textup{(}respectively, $n+1-i$\textup{)}.
\end{enumerate}
\end{lemma}

\begin{lemma}\label{lem_index}
  Let $\mathbf{i}=(i_1,\dots, i_{\bar n})\in\Rn{n+1}$.
  Then there are unique integers $d_1,\dots,d_n$ such that
  \begin{equation}\label{eq_D}
	1 \leq d_1 < \dots < d_n \leq \bar{n} , \quad \quad (i_{d_1}, \dots, i_{d_n})= (n,n-1, \dots, 2, 1),
  \end{equation}
  and there are unique integers $a_1,\dots,a_n$ such that
  \begin{equation}\label{eq_A}
	1 \leq a_1 < \dots < a_n \leq \bar{n} , \quad \quad (i_{a_1}, \dots, i_{a_n})= (1,2, \dots,n-1, n).
  \end{equation}
 Moreover, $|\{a_1,\dots,a_n\}\cap\{d_1,\dots,d_n\}|=1$.
\end{lemma}
\begin{proof}
  We will use the wiring diagram $G(\mathbf{i})$ of $\mathbf{i}$. Let
  $d_1<\dots<d_n$ be the row indices of the crossings in the $(n+1)$st wire of
  $G(\mathbf{i})$. Then, by Lemma~\ref{lemma_properties_of_longest_elt},
  $d_1,\dots,d_n$ satisfy \eqref{eq_D}. Similarly, let $a_1<\dots<a_n$ be the
  row indices of the crossings in the $1$st wire of $G(\mathbf{i})$. Then, by
  Lemma~\ref{lemma_properties_of_longest_elt}, $a_1,\dots,a_n$ satisfy
  \eqref{eq_A}. Since the $1$st and the $(n+1)$st wires meet exactly once, we
  get $|\{a_1,\dots,a_n\}\cap\{d_1,\dots,d_n\}|=1$. It remains to show that
  these integers are unique.

  To show the uniqueness of $d_1,\dots,d_n$ satisfying \eqref{eq_D},
  suppose that $d_1',\dots,d_n'$ are integers such that $d_j\ne d_j'$ for some
  $j$ and
\[
	1 \leq d_1' < \dots < d_n' \leq \bar{n} , \quad \quad (i_{d_1'}, \dots, i_{d_n'})= (n,n-1, \dots, 2, 1).
\]
Let $k$ be the smallest integer such that $d_{j} = d_{j}'$ for $j< k$ and $d_k \neq d_k'$.
Since the $(n+1)$st wire passes through the crossing in row $d_{k-1}$ and column
$n+2-k$, the crossing in row $d_{k}$ and column $n+1-k$, and the crossing in row
$d_{k+1}$ and column~$n-k$, there is no crossing in row $j$ and column $n+1-k$
for all $d_{k-1}<j<d_{k+1}$ with $j\ne d_k$. Since $d_{k-1}=d_{k-1}'<d_k'$, this
shows $d_{k+1} < d_k'$. See Figure~\ref{fig:dk}.

By the same argument, we can deduce that $d_{j+1}<d_j'$ for
$j=k,k+1,\dots,n-1$. In particular we have $d_n<d_{n-1}'<d_n'$. However, by
the definition of $d_n$, the crossing in row $d_n$ and column $1$ is the lowest
crossing in this column, which is a contradiction to $d_n<d_n'$. This shows that
there are no such integers $d_1',\dots,d_n'$, so the uniqueness of
$d_1,\dots,d_n$ is proved.

Similarly, we can show the uniqueness of $a_1,\dots,a_n$, and the proof is
completed.
\end{proof}

\begin{figure}
  \centering
		\begin{tikzpicture}[scale = 0.7]
\tikzset{every node/.style = {font = \footnotesize}}
\tikzset{red line/.style = {line width=0.5ex, red, semitransparent}}
\tikzset{blue line/.style = {line width=0.5ex, blue, semitransparent}}


\draw[red] (1,0)--(1,0.5)--(2,1)--(2,2)--(3,2.5)--(3,5)--(4,5.5)--(4,7);
\draw (3,5.5)--(4,5);
\draw (2,2.5)--(3,2);
\draw (1,1)--(2,0.5);

\draw (2,-1)--(3,-1.5);
\draw (3,-1)--(2,-1.5);

\draw[gray,dotted] (-1,5.25)--(3.5,5.25) node[at start, left,  color=black]  {row $d_{k-1} = d_{k-1}'$};
\draw[gray, dotted] (-1,2.25)--(2.5,2.25) node[at start, left, color = black] {row $d_k$};
\draw[gray, dotted] (-1, 0.75)--(1.5,0.75) node[at start, left, color = black] {row $d_{k+1}$};
\draw[gray, dotted] (-1, -1.25)--(2.5,-1.25) node[at start, left, color=black] {row $d_k'$};

\draw[gray,dotted] (2.5,-1.25)--(2.5,7) node[at end, above, color=black] {column $n+1-k$};
\end{tikzpicture}
 \caption{An illustration of a part of the $(n+1)$st wire (colored in red) in a wiring diagram and the
   crossings in rows $d_{k-1},d_k,d_{k+1}$, and $d_k'$.}
  \label{fig:dk}
\end{figure}

The previous lemma can be restated in terms of word posets as follows.

\begin{proposition}\label{prop_index}
  Let $P\in\P(w_0^{(n+1)})$. Then $P$ contains a unique chain
  $\D(P)$ such that
\[
\D(P) = \{d_1<_{P} d_{2}<_{P} \dots <_{P} d_n\}
\]
and $f_P(d_i)=n+1-i$ for all $1\le i \le n$. Similarly, $P$ contains a unique
chain $\A(P)$ such that
\[
\A(P) = \{a_1<_P a_{2}<_P \dots <_P a_{n}\}
\]
and $f_P(a_i)=i$ for all $1\le i \le n$. Moreover, we have $|\D(P) \cap \A(P)| = 1$.
\end{proposition}

\begin{proof}
	Let $P\in\P(w_0^{(n+1)})$. By Proposition~\ref{prop:i=P}, there exists
  $\mathbf i \in \RR(w_0^{(n+1)})$ satisfying~$P_{\mathbf i}\sim P$.
  Accordingly, it is enough to prove the statements for the word poset
  $P_{\mathbf i}$ for an arbitrary $\mathbf{i}=(i_1,\dots, i_{\bar
    n})\in\Rn{n+1}$. Then, using the map in Proposition~\ref{prop:i=P}, the
  statements for $\Pi$ that we need to prove can be reformulated as
  the statements for $\mathbf{i}$ in Lemma~\ref{lem_index}. Therefore
  the proof follows from this lemma.
\end{proof}

We call $\D(P)$ the \emph{descending chain} of $P$ and $\A(P)$ the
\emph{ascending chain} of $P$. We now define an important notion in this paper
called indices.

\begin{definition}\label{def_index}
  Let $P\in\P(w_0^{(n+1)})$. The \defi{$\D$-index} of $P$, denoted by
  $\ind_{\D}(P)$, is the number of elements in $P$ above the descending chain
  $\D(P)$ in the Hasse diagram of $P$. Similarly, the \defi{$\A$-index} of $P$,
  denoted by $\ind_{\A}(P)$, is the number of elements in $P$ above the
  ascending chain in the Hasse diagram of $P$. More precisely, if $\D(P) =
  \{d_1<_P d_{2}<_P \dots <_P d_n\}$ and $\A(P) = \{a_1<_P a_{2}<_P \dots <_P
  a_{n}\}$, then
  \begin{align*}
   \ind_{\D}(P) &= \sum_{i=1}^n \#\{k\in[\barn]: k>_P d_i \mbox{ and } f_P(k)=f_P(d_i) \},\\
   \ind_{\A}(P) &= \sum_{i=1}^n \#\{k\in[\barn]: k>_P a_i \mbox{ and } f_P(k)=f_P(a_i) \}.
  \end{align*}
  For $\mathbf{i}\in\Rn{n+1}$, we also define
\[
\ind_{\D}(\mathbf{i}) =  \ind_{\D}(\Pi), \qquad
\ind_{\A}(\mathbf{i}) =  \ind_{\A}(\Pi).
\]
\end{definition}

Note that in the above definition the summand $\#\{k\in[\barn]: k>_{P} d_i \mbox{
  and } f_P(k)=f_p(d_i) \}$ is the number of elements of $P$ above the element
$d_i$ of $\D(P)$ in column $f_P(d_i)=n+1-i$ in the Hasse diagram of $P$.

In \cite[Definition 3.4]{CKLP} the indices $\ind_{\D}(\mathbf{i})$ and
$\ind_{\A}(\mathbf{i})$ of $\mathbf{i}\in\Rn{n+1}$ are defined without using
word posets. It is not hard to see that the two definitions are equivalent.

\begin{example}\label{example_index_132132}
	Continuing Example~\ref{example_poset}, let $\mathbf i = (1,2,1,3,2,1)$ and
  $\mathbf j = (1,3,2,1,3,2)$. Then we have that
  \[
  \begin{split}
  \A(P_{\mathbf i}) = 1 <_{P_{\mathbf i}} 2 <_{P_{\mathbf i}} 4, \quad \D(P_{\mathbf i}) = 4 < _{P_{\mathbf i}}5 <_{P_{\mathbf i}} 6, \\
  \A(P_{\mathbf j}) = 1 <_{P_{\mathbf j}} 3 <_{P_{\mathbf j}} 5, \quad \D(P_{\mathbf j}) = 2 <_{P_{\mathbf j}} 3 <_{P_{\mathbf j}} 4.
  \end{split}
  \]
  The chains $\D(P)$ and $\A(P)$ for $P=\Pi$
  (left) and $P=\Pj$ (right) are shown as follows.
	
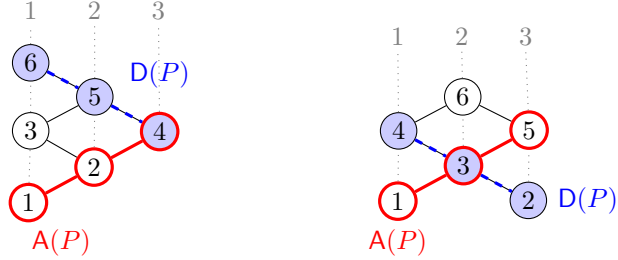
\begin{center}
\begin{tikzpicture}[auto,node distance=0.1cm and 0.5cm, inner sep=2pt,
halo/.style={line join=round,
	double,line cap=round,double distance=#1,shorten >=-#1/2,shorten <=-#1/2},
halo/.default=7mm]
\tikzstyle{state}=[draw, circle]
\node[state, fill=blue!20] (1) {$6$};
\node[state, , fill=blue!20] (2) [below right=of 1] {$5$};
\node[state] (3) [below left =of 2] {$3$};
\node[state, very thick, draw=red, fill=blue!20, very thick, draw=red] (4) [below right =of 2] {$4$};
\node[state, very thick, draw=red] (5) [below left =of 4] {$2$};
\node[state, very thick, draw=red] (6) [below left =of 5] {$1$};

\draw (1)-- (2)
(2)--(3)
(2)--(4)
(3)--(5)
(4)--(5)
(5)--(6);

\draw[very thick, red] (4)--(5)--(6) node [midway, below=0.5cm] {$\A(P)$};

\node (4above) [below of = 4, node distance =0.55cm] {};
\node (1below) [above of =1, node distance = 0.55cm] {};


		\draw[dotted, gray] (6)--(3)--(1)--(0,0.5) node[above] {$1$};
\draw[dotted, gray] (5)--(2)--(0.85,0.5) node[above] {$2$};
\draw[dotted, gray] (4)--(1.7, 0.5) node[above] {$3$};

    \draw[very thick, dashed, blue] (1)--(2)--(4)  node[above=0.5cm] {$\D(P)$};

\end{tikzpicture}
\hspace{2cm}
\begin{tikzpicture}[auto,node distance=0.1cm and 0.5cm, inner sep=2pt,
halo/.style={line join=round,
	double,line cap=round,double distance=#1,shorten >=-#1/2,shorten <=-#1/2},
halo/.default=7mm]
\tikzstyle{state}=[draw, circle]

\tikzstyle{state}=[draw, circle]
\node[state,very thick, draw=red] (1) {$1$};
\node[state, very thick, draw=red, , fill=blue!20] (3) [above right = of 1] {$3$};
\node[state,  fill=blue!20] (2) [below right = of 3] {$2$};
\node[state, fill=blue!20] (4) [above left = of 3] {$4$};
\node[state,very thick, draw=red] (5) [above right = of 3] {$5$};
\node[state] (6) [above right = of 4] {$6$};

\draw (1)--(3)
(3)--(2)
(3)--(4)
(3)--(5)
(4)--(6)
(6)--(5);

\node (2above) [below  of = 2, node distance =0.55cm] {};
\node (4below) [above of =4, node distance = 0.55cm] {};

\draw[very thick, red] (1)--(3)--(5);
\node [below of = 1, node distance = 0.55cm, red] {$\A(P)$};

\draw[dotted, gray] (1)--(4)--(0,2) node[above] {$1$};
\draw[dotted, gray] (3)--(6)--(0.85,2) node[above] {$2$};
\draw[dotted, gray] (2)--(5)--(1.7, 2) node[above] {$3$};

    \draw[very thick, dashed, blue] (4)--(3)--(2)  node[right=0.3cm] {$\D(P)$};

\end{tikzpicture}
\end{center}
Counting the number of elements above $\A(P)$ and $\D(P)$, we obtain
	\[
	\ind_\D(\mathbf{i})=0,\quad \ind_\A(\mathbf{i}) = 3, \quad
	\ind_\D(\mathbf{j}) = 2,\quad \ind_\A(\mathbf{j}) = 2.
	\]
\end{example}

The following lemma shows that the elements of $P\in\P(w_0^{(n+1)})$ in each
column form a chain.

\begin{lemma}\label{lem:chain}
  Let $P\in\P(w_0^{(n+1)})$. Then $\{ x\in P: f_P(x) = i\}$ is a chain in $P$
  for each $i\in[n]$.
\end{lemma}
\begin{proof}
  Since $P\in\P(w_0^{(n+1)})$, there is a reduced word
  $\mathbf{i}=(i_1,i_2,\dots,i_{\barn})\in\Rn{n+1}$ such that $P_{\mathbf{i}}
  \sim P$.
  Therefore it suffices to show that $\{ x\in \Pi: f_{\Pi}(x) = i\}$ is a chain.
  Consider $x,y\in\Pi$ with $\ffi(x)=\ffi(y)=i$ and $x<_\Z y$. This means
  $i_x=i_y=i$. Since $\mathbf{i}$ is reduced there must be an integer $z$ such
  that $x<z<y$ and $i_z\in\{i-1,i+1\}$. Then $x<_{\Pi}z<_{\Pi}y$ and therefore
  $x<_{\Pi}y$. Since any two elements $x,y\in \Pi$ with $\ffi(x)=\ffi(y)$ are
  comparable, $\{ x\in \Pi: f_{\Pi}(x) = i\}$ is a chain and the lemma is
  proved.
\end{proof}

Now we define two operations, called contractions, on word posets.

\begin{definition}
	Let $P\in\P(w_0^{(n+1)})$ and let
  \begin{align*}
\D(P) &= \{d_1<_P d_{2}<_P \dots <_P d_n\},\\
\A(P) &= \{a_1<_P a_{2}<_P \dots <_P a_n\}.
  \end{align*}
  \begin{enumerate}
\item  The \emph{$\D$-contraction} of $P$ is the word poset $C_\D(P)$ obtained from
  $P$ by removing the descending chain $\D(P)$ with the function
  $f_{C_\D(P)} \colon C_\D(P)\to\Z_{>0}$ given by
\[
f_{C_\D(P)}(k) = 
\begin{cases}
f_P(k) & \mbox{if $k\in I_\D(P)$},\\
f_P(k)-1 & \mbox{if $k\in C_\D(P)\setminus I_\D(P)$},
\end{cases}
\]
where $ I_\D(P) = \{k\in P: k<_P d_i \text{ and } f_P(k) = f_P(d_i) \text{ for
  some } i \in [n] \}$. Here, the poset structure of $C_\D(P)$ is induced
  from that of $P$, i.e., for $x,y\in C_\D(P)$ we have $x<_{C_\D(P)}y$ if and
  only if $x<_{P}y$.
\item The \emph{$\A$-contraction} of $P$ is the word poset $C_\A(P)$ obtained from $P$
by removing the ascending chain $\A(P)$ with the function
$f_{C_\A(P)} \colon C_\A(P)\to\Z_{>0}$ given by
\[
f_{C_\A(P)}(k) = 
\begin{cases}
f_P(k)-1 & \mbox{if $k\in  I_\A(P)$},\\
f_P(k) & \mbox{if $k\in C_\D(P)\setminus I_\A(P)$},
\end{cases}
\]
where
$ I_\A(P) = \{k\in P: k<_P a_i \text{ and } f_P(k) =f_P(a_i) \text{ for some} i \in [n]\}$. 
Here, the poset structure of $C_\A(P)$ is induced
  from that of $P$.
\end{enumerate}
\end{definition}

Observe that $I_\D(P)$ (respectively, $I_\A(P)$) is the ideal consisting of the
elements below the descending chain $\D(P)$ (respectively, ascending chain
$\A(P)$) in the Hasse diagram of $P$. We call $I_\D(P)$ (respectively,
$I_\A(P)$) the \emph{$\D$-contraction ideal} (respectively,
\emph{$\A$-contraction ideal}) of $P$. Here, an \emph{ideal} of a poset $P$
means a subset $I$ of $P$ with the property that $x\in I$ and $y<_P x$ imply
$y\in I$.

One may consider $C_\D(P)$ as the word
poset whose Hasse diagram is obtained from that of $P$ by removing the
descending chain $\D(P)$ and shifting the part $(P\setminus \D(P)) \setminus
I_\D(P)$ above $\D(P)$ to the left by one column. Similarly one may consider
$C_\A(P)$ as the word poset whose Hasse diagram is obtained from that of $P$ by
removing the ascending chain $\A(P)$ and shifting the part $I_\A(P)$ below
$\A(P)$ to the left by one column. See Figure~\ref{fig_poset_example_Rn6}.

\begin{example}
	Let $\mathbf i = (4,3,4,2,3,4,1,2,5,4,3,2,1,4,5) \in \Rn{6}$. The Hasse
	diagrams of $\Pi$, $C_\A(\Pi)$, and $C_\D(\Pi)$ are
	shown in Figure~\ref{fig_poset_example_Rn6}.
	Note that  $\ind_{\A}(\Pi) = 2$ and $\ind_{\D}(\Pi) = 2$.
	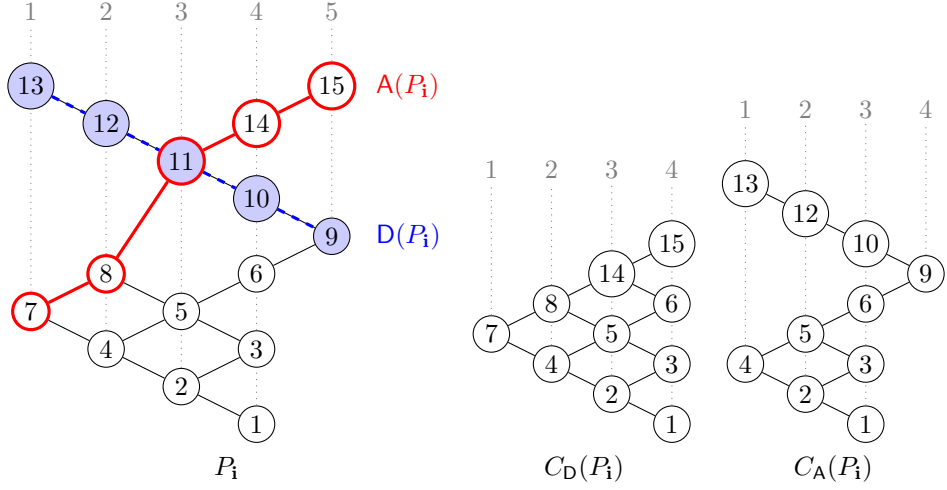
\begin{figure}
		\begin{tabular}{ccc}
			\begin{tikzpicture}[auto,node distance=0.1cm and 0.5cm, inner sep=2pt,
			halo/.style={line join=round,
				double,line cap=round,double distance=#1,shorten >=-#1/2,shorten <=-#1/2},
			halo/.default=7mm]
			\tikzstyle{state}=[draw, circle]
			
			\node[state] (1) at (0,0) {$1$};
			\node[state] (2) at (-1,0.5) {$2$};
			\node[state] (3) at (0,1) {$3$};
			\node[state] (4) at (-2, 1)  {$4$};
			\node[state] (5) at (-1,1.5)  {$5$};
			\node[state] (6) at (0,2){$6$};
			\node[state, very thick, draw=red] (7) at (-3, 1.5) {$7$};
			\node[state, very thick, draw=red] (8) at (-2,2) {$8$};
			\node[state, fill=blue!20] (9) at (1, 2.5) {$9$};
			\node[state, fill=blue!20] (10) at (0,3) {$10$};
			\node[state, very thick, draw=red, fill=blue!20] (11) at (-1,3.5) {$11$};
			\node[state, fill=blue!20] (12) at (-2,4) {$12$};
			\node[state, fill=blue!20] (13) at (-3, 4.5) {$13$};
			\node[state, very thick, draw=red] (14) at (0, 4) {$14$};
			\node[state, very thick, draw=red] (15) at (1, 4.5) {$15$};
			
			\draw (1)--(2) (2)--(3) (2)--(4) (3)--(5) (4)--(5) (4)--(7)--(8)--(5)
			(5)--(6)--(9)--(10)--(11)--(12)--(13)
			(8)--(11)--(14)--(15);
			
			\draw[very thick, red] (7)--(8)--(11)--(14)--(15) node[right = 0.5cm] {$\A(P_{\mathbf i})$};
			
			\node (9above) [below  of = 9, node distance =0.55cm] {};
			\node (13below) [above of =13, node distance = 0.55cm] {};
			

    \draw[very thick, dashed, blue] (13)--(12)--(11)--(10)--(9) node[right=0.5cm] {$\D(P_{\mathbf i})$}; 
    			
			
		\draw[dotted, gray] (7)--(13)--(-3,5.3) node[above] {$1$};
\draw[dotted, gray] (4)--(8)--(12)--(-2, 5.3) node[above] {$2$};
\draw[dotted, gray] (2)--(5)--(11)--(-1,5.3) node[above] {$3$};
\draw[dotted, gray] (1)--(3)--(6)--(10)--(14)--(0, 5.3) node[above] {$4$};
\draw[dotted, gray] (9)--(15)--(1,5.3) node[above] {$5$};

			\end{tikzpicture}
			&
			\begin{tikzpicture}[auto,node distance=0.1cm and 0.5cm, inner sep=2pt, scale = 0.8]
			\tikzstyle{state}=[draw, circle]
			
			\node[state] (1) at (0,0) {$1$};
\node[state] (2) at (-1,0.5) {$2$};
\node[state] (3) at (0,1) {$3$};
\node[state] (4) at (-2, 1)  {$4$};
\node[state] (5) at (-1,1.5)  {$5$};
\node[state] (6) at (0,2){$6$};
\node[state ] (7) at (-3, 1.5) {$7$};
\node[state] (8) at (-2,2) {$8$};

			\node[state] (9) at (-1,2.5) {$14$};
			\node[state] (10) at (0,3)  {$15$};
			
			\draw (1)--(2)--(3)
			(4)--(5)--(6)
			(7)--(8)--(9)--(10)
			(2)--(4)--(7)
			(3)--(5)--(8)
			(6)--(9);
	
			\draw[dotted, gray] (7)--(-3,4) node[above] {$1$};
	\draw[dotted, gray] (4)--(8)--(-2, 4) node[above] {$2$};
	\draw[dotted, gray] (2)--(5)--(9)--(-1,4) node[above] {$3$};
	\draw[dotted, gray] (1)--(3)--(6)--(10)--(0,4) node[above] {$4$};		
			\end{tikzpicture}
			&
			\begin{tikzpicture}[auto,node distance=0.1cm and 0.5cm, inner sep=2pt, scale = 0.8]
			\tikzstyle{state}=[draw, circle]
			
			\node[state] (1) at (0,0) {$1$};
\node[state] (2) at (-1,0.5) {$2$};
\node[state] (3) at (0,1) {$3$};
\node[state] (4) at (-2, 1)  {$4$};
\node[state] (5) at (-1,1.5)  {$5$};
\node[state] (6) at (0,2){$6$};			
\node[state ] (7) at (1, 2.5) {$9$};
\node[state] (8) at (0,3) {$10$};			
			
			\node[state] (9) at (-1,3.5) {$12$};
			\node[state] (10) at (-2,4) {$13$};
			
			\draw (1)--(2)--(3)--(5)--(6)--(7)--(8)--(9)--(10)
			(2)--(4)--(5);
			
\draw[dotted, gray] (4)--(10)--(-2,5) node[above] {$1$};
\draw[dotted, gray] (2)--(5)--(9)--(-1, 5) node[above] {$2$};
\draw[dotted, gray] (1)--(3)--(6)--(8)--(0,5) node[above] {$3$};				
\draw[dotted, gray] (7)--(1,5) node[above] {$4$};
			\end{tikzpicture} 
			\\
			$\Pi$ & $C_\D(\Pi)$ & $C_\A(\Pi)$
		\end{tabular}
		\caption{The Hasse diagrams of $\Pi$, $C_\D(\Pi)$, and
			$C_\A(\Pi)$ for $\mathbf i = (4,3,4,2,3,4,1,2,5,4,3,2,1,4,5)$. In
			each diagram, the element $k$ is in column $\ffi(k)$,
			$f_{C_\D(\Pi)}(k)$, or $f_{C_\A(\Pi)}(k)$.}
		\label{fig_poset_example_Rn6}
	\end{figure}
\end{example}

Now we define the reverse operations of the contractions.

\begin{definition}\label{def:extensions}
	Let $P\in\P(w_0^{(n)})$ and let $I$ be an ideal of $P$. 
\begin{enumerate}
\item  The \emph{$\D$-extension} of $P$ with respect to $I$ is the word poset
  $E_\D(P,I)$ which is a poset on $P\sqcup \{d_1,d_2,\dots,d_n\}$ with a function
  $f_{E_\D(P,I)} \colon E_\D(P,I)\to \Z_{>0}$ defined as follows.
  \begin{itemize}
  \item The function $f_{E_\D(P,I)} \colon E_\D(P,I)\to \Z_{>0}$ is given by 
\[
  f_{E_\D(P,I)}(x) =
  \begin{cases}
    f_P(x) & \mbox{if $x\in  I$},\\
    n+1-i & \mbox{if $x=d_i$},\\
    f_P(x)+1 & \mbox{if $x\in P\setminus I$}.
  \end{cases}
\]
\item The covering relations of the poset $E_\D(P,I)$ on $P\sqcup
  \{d_1,d_2,\dots,d_n\}$ are given by $x\lessdot_{E_\D(P,I)}y$ if and only if
  one of the following conditions holds:
  \begin{enumerate}
  \item $x,y\in I$ and $x\lessdot_P y$,
  \item $x,y\in P\setminus I$ and $x\lessdot_P y$,
  \item $x=d_i$ and $y=d_{i+1}$ for some $i\in[n-1]$,
  \item $x$ is a maximal element of $I$ in $P$, $y\in \{d_1,\dots,d_n\}$, and
    \[
    |f_{E_\D(P,I)}(x)-f_{E_\D(P,I)}(y)|=1\text{, or}
    \]
  \item $x\in \{d_1,\dots,d_n\}$, $y$ is a minimal element of $P\setminus I$ in
    $P$, and 
    \[
    |f_{E_\D(P,I)}(x)-f_{E_\D(P,I)}(y)|=1.
    \]
  \end{enumerate}
  \end{itemize}

\item   The \emph{$\A$-extension} of $P$ with respect to $I$ is the word poset
  $E_\A(P,I)$ which is a poset on $P\sqcup \{a_1,a_2,\dots,a_n\}$ with a function
  $f_{E_\A(P,I)}\colon E_\A(P,I)\to \Z_{>0}$ defined as follows.
  \begin{itemize}
  \item The function $f_{E_\A(P,I)} \colon E_\A(P,I)\to \Z_{>0}$ is given by 
\[
  f_{E_\A(P,I)}(x) =
  \begin{cases}
    f_P(x)+1 & \mbox{if $x\in  I$},\\
    i & \mbox{if $x=a_i$},\\
    f_P(x) & \mbox{if $x\in P\setminus I$}.
  \end{cases}
\]
\item The covering relations of the poset $E_\A(P,I)$ on $P\sqcup
  \{a_1,a_2,\dots,a_n\}$ are given by $x\lessdot_{E_\A(P,I)}y$ if and only if
  one of the following conditions holds:
  \begin{enumerate}
  \item $x,y\in I$ and $x\lessdot_P y$,
  \item $x,y\in P\setminus I$ and $x\lessdot_P y$,
  \item $x=a_i$ and $y=a_{i+1}$ for some $i\in[n-1]$,
  \item $x$ is a maximal element of $I$ in $P$, $y\in \{a_1,\dots,a_n\}$, and
    \[
   |f_{E_\A(P,I)}(x)-f_{E_\A(P,I)}(y)|=1\text{, or}
   \]
  \item $x\in \{a_1,\dots,a_n\}$, $y$ is a minimal element of $P\setminus I$ in
    $P$, and 
    \[
    |f_{E_\A(P,I)}(x)-f_{E_\A(P,I)}(y)|=1.\]
  \end{enumerate}
  \end{itemize}
\end{enumerate}
\end{definition}

The extensions are the reverse operations of the contractions in the following
sense.

\begin{proposition}\label{prop:EC}
	Let $P\in\P(w_0^{(n+1)})$. Then
  \begin{align*}
P &\sim E_\D(C_\D(P), I_\D(P)),\\
P &\sim E_\A(C_\A(P), I_\A(P)).
  \end{align*}
\end{proposition}
\begin{proof}
  This is straightforward to check using the definitions of contractions and
  extensions. We omit the details.
\end{proof}
In \cite[Definition~3.6]{CKLP}, the $\D$-contraction $C_\D(\mathbf{i})$ and the
$\A$-contraction~$C_\A(\mathbf{i})$ of a reduced word $\mathbf i \in \Rn{n+1}$
are defined using wiring diagrams. For $\mathbf i \in \Rn{n+1}$, let $G(\mathbf
i)$ be the corresponding wiring diagram. Removing the $(n+1)$st wire from
$G(\mathbf i)$ and shifting the part below this wire to the left by one, we get
a new wiring diagram such that the $j$th wire travels from the $j$th starting
point to the $(n+1-j)$th ending point. Since the number of crossings decreases
by $n$, the new wiring diagram represents a reduced word in $\Rn{n}$. Similarly,
removing the $1$st wire from $G(\mathbf{i})$ also produces the wiring diagram of
a reduced word in $\Rn{n}$. One can check that $C_\D(\Pi)\sim
P_{C_\D(\mathbf{i})}$ and $C_\A(\Pi)\sim P_{C_\A(\mathbf{i})}$. This leads us to
the following result.
\begin{proposition}
  Let $P\in\P(w_0^{(n+1)})$. Then $C_\D(P)$ and $C_\A(P)$ are word posets in~$\P(w_0^{(n)})$.
\end{proposition}

Finally, we define the $\delta$-index of a word poset.

\begin{definition}
  For a sequence $\ad = (\ad_1,\dots,\ad_{n-1}) \in \{\A,\D\}^{n-1}$, the {\em
    $\ad$-index} of~$P\in\P(w_0^{(n+1)})$ is the integer vector
  $\mathrm{ind}_\ad(P) = (I_1,\dots,I_{n-1})$ defined by
\[
I_k \coloneqq \ind_{\ad_k}(C_{\ad_{k+1}} \circ C_{\ad_{k+2}} \circ  \cdots \circ C_{\ad_{n-1}}(P)),
\qquad 1\le k\le n-1,
\]
where $I_{n-1} = \ind_{\ad_{n-1}}(P)$. The {\em $\ad$-index} of $\mathbf i \in \Rn{n+1}$
is defined by $\mathrm{ind}_\ad({\bf i}) = \mathrm{ind}_\ad(\Pi)$.
\end{definition}

\begin{example}
  Let $\mathbf i = (4,3,4,2,3,4,1,2,5,4,3,2,1,4,5) \in \Rn{6}$. Then, for a
  sequence $\ad = (\A,\A,\A,\A) \in \{\A,\D\}^4$, we obtain $\ind_{\ad}(\mathbf
  i) = (1,2,3,2)$ as shown in Figure~\ref{fig_example_434234125432145}.
	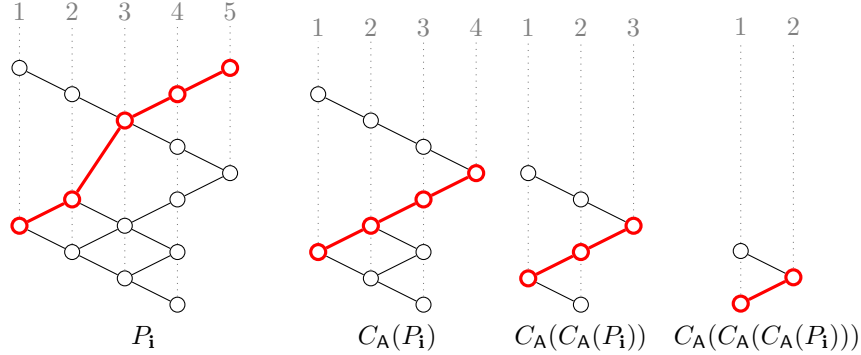
\begin{figure}
	\begin{tabular}{cccc}
		\begin{tikzpicture}[auto,node distance=0.1cm and 0.5cm, inner sep=2pt,
		halo/.style={line join=round,
			double,line cap=round,double distance=#1,shorten >=-#1/2,shorten <=-#1/2},
		halo/.default=7mm, 
		scale = 0.7]
		\tikzstyle{state}=[draw, circle]
		
		\node[state] (1) at (0,0) {};
		\node[state] (2) at (-1,0.5) {};
		\node[state] (3) at (0,1) {};
		\node[state] (4) at (-2, 1)  {};
		\node[state] (5) at (-1,1.5)  {};
		\node[state] (6) at (0,2){};
		\node[state, very thick, draw=red] (7) at (-3, 1.5) {};
		\node[state, very thick, draw=red] (8) at (-2,2) {};
		\node[state ] (9) at (1, 2.5) {};
		\node[state ] (10) at (0,3) {};
		\node[state, very thick, draw=red ] (11) at (-1,3.5) {};
		\node[state ] (12) at (-2,4) {};
		\node[state ] (13) at (-3, 4.5) {};
		\node[state, very thick, draw=red] (14) at (0, 4) {};
		\node[state, very thick, draw=red] (15) at (1, 4.5) {};
		
		\draw (1)--(2) (2)--(3) (2)--(4) (3)--(5) (4)--(5) (4)--(7)--(8)--(5)
		(5)--(6)--(9)--(10)--(11)--(12)--(13)
		(8)--(11)--(14)--(15);
		
		\draw[very thick, red] (7)--(8)--(11)--(14)--(15) node[right = 0.5cm] {};
		
		\node (9above) [below  of = 9, node distance =0.55cm] {};
		\node (13below) [above of =13, node distance = 0.55cm] {};

		\draw[dotted, gray] (7)--(13)--(-3,5.3) node[above] {$1$};
		\draw[dotted, gray] (4)--(8)--(12)--(-2, 5.3) node[above] {$2$};
		\draw[dotted, gray] (2)--(5)--(11)--(-1,5.3) node[above] {$3$};
		\draw[dotted, gray] (1)--(3)--(6)--(10)--(14)--(0, 5.3) node[above] {$4$};
		\draw[dotted, gray] (9)--(15)--(1,5.3) node[above] {$5$};

		\end{tikzpicture}
		&
	
		\begin{tikzpicture}[auto,node distance=0.1cm and 0.5cm, inner sep=2pt, scale = 0.7]
		\tikzstyle{state}=[draw, circle]
		
		\node[state] (1) at (0,0) {};
		\node[state] (2) at (-1,0.5) {};
		\node[state] (3) at (0,1) {};
		\node[state, very thick, draw=red] (4) at (-2, 1)  {};
		\node[state, very thick, draw=red] (5) at (-1,1.5)  {};
		\node[state, very thick, draw=red] (6) at (0,2){};			
		\node[state , very thick, draw=red] (7) at (1, 2.5) {};
		\node[state] (8) at (0,3) {};			
		
		\node[state] (9) at (-1,3.5) {};
		\node[state] (10) at (-2,4) {};
		
		\draw (1)--(2)--(3)--(5)--(6)--(7)--(8)--(9)--(10)
		(2)--(4)--(5);
		
			\draw[very thick, red] (4)--(5)--(6)--(7);
		
		\draw[dotted, gray] (4)--(10)--(-2,5) node[above] {$1$};
		\draw[dotted, gray] (2)--(5)--(9)--(-1, 5) node[above] {$2$};
		\draw[dotted, gray] (1)--(3)--(6)--(8)--(0,5) node[above] {$3$};				
		\draw[dotted, gray] (7)--(1,5) node[above] {$4$};
		\end{tikzpicture} 
		&
				\begin{tikzpicture}[auto,node distance=0.1cm and 0.5cm, inner sep=2pt, scale = 0.7]
		\tikzstyle{state}=[draw, circle]
		
		\node[state] (1) at (-1,0) {};
		\node[state, very thick, draw=red] (2) at (-2,0.5) {};
		\node[state, very thick, draw=red] (3) at (-1,1) {};

		\node[state, very thick, draw=red] (8) at (0,1.5) {};			
		
		\node[state] (9) at (-1,2) {};
		\node[state] (10) at (-2,2.5) {};
		
		\draw (1)--(2)--(3)--(8)--(9)--(10);
		
		\draw[very thick, red] (2)--(3)--(8);
		
		\draw[dotted, gray] (2)--(10)--(-2,5) node[above] {$1$};
		\draw[dotted, gray] (1)--(3)--(9)--(-1, 5) node[above] {$2$};
		\draw[dotted, gray] (8)--(0,5) node[above] {$3$};				
		\end{tikzpicture} 
		&
						\begin{tikzpicture}[auto,node distance=0.1cm and 0.5cm, inner sep=2pt, scale = 0.7]
		\tikzstyle{state}=[draw, circle]
		
		\node[state, very thick, draw=red] (1) at (-2,0) {};
		
		\node[state, very thick, draw=red] (9) at (-1,0.5) {};
		\node[state] (10) at (-2,1) {};
		
		\draw (1)--(9)--(10);
		
		\draw[very thick, red] (1)--(9);
		
		\draw[dotted, gray] (1)--(10)--(-2,5) node[above] {$1$};
		\draw[dotted, gray] (9)--(-1, 5) node[above] {$2$};
		\end{tikzpicture} 
		\\
		$\Pi$ &  $C_\A(\Pi)$ & $C_\A(C_\A(\Pi))$ & $C_\A(C_\A(C_\A(\Pi)))$
	\end{tabular}
	\caption{The Hasse diagrams of $\Pi$, 
		$C_\A(\Pi)$, $C_\A(C_\A(\Pi))$, and $C_\A(C_\A(C_\A(\Pi)))$ for $\mathbf i = (4,3,4,2,3,4,1,2,5,4,3,2,1,4,5)$.}
	\label{fig_example_434234125432145}
\end{figure}
\end{example}

\section{Gelfand--Cetlin type reduced words}
\label{secGelfandCetlinTypeCommutationClasses}
In this section, we introduce Gelfand--Cetlin type reduced words. We give a
recursive formula for the number of such reduced words using standard Young
tableaux of shifted shapes in Theorem~\ref{thm_SYT}. 

We first define Gelfand--Cetlin type word posets and Gelfand--Cetlin type
reduced words.

\begin{definition}
  A word poset $P\in\P(w_0^{(n+1)})$ is of \emph{Gelfand--Cetlin type} if there
  exists $\delta \in \{\A, \D\}^{n-1}$ such that $\ind_{\delta}(P) = (0,0,\dots,0)
  \in \Z^{n-1}$. Denote by $\PGC(n)$ the set of Gelfand--Cetlin type word posets in
  $\P(w_0^{(n+1)})$ and let $[\PGC(n)]=\PGC(n)/\sim$.
\end{definition}

\begin{definition}
  A reduced word $\mathbf i \in \Rn{n+1}$ is of \defi{Gelfand--Cetlin type} if
  $\Pi$ is of Gelfand--Cetlin type. A commutation class~$[\mathbf i]$ is of
  \defi{Gelfand--Cetlin type} if it contains a Gelfand--Cetlin type reduced
  word. Denote by $\RGC(n)$ the set of Gelfand--Cetlin type reduced words in
  $\Rn{n+1}$ and let $[\RGC(n)]=\{[\mathbf{i}]:\mathbf{i}\in\RGC(n)\}$.
\end{definition}

\begin{remark}
  One can deduce from~\cite[Theroem~A]{CKLP} that for $\mathbf i \in
  \RR(w_0^{(n+1)})$, the string polytope $\Delta_\mathbf i(\lambda)$ is
  combinatorially equivalent to a full dimensional Gelfand--Cetlin polytope of
  rank $n$ if and only if $\mathbf i$ is of Gelfand--Cetlin type, see the
  proof of Corollary~\ref{cor_main2}. This is why we say that such word posets
  and reduced words are of Gelfand--Cetlin type.
\end{remark}
The following proposition easily follows from Proposition~\ref{prop:i=P} and the
definitions of $\PGC(n)$ and $\RGC(n)$.

\begin{proposition}\label{prop:RGC=PGC}
  The map $[\mathbf{i}]\mapsto [\Pi]$ is a bijection from $[\RGC(n)]$ to
  $[\PGC(n)]$. Accordingly,
\[
|[\RGC(n)]|=|[\PGC(n)]|.
\]
\end{proposition}

We note that if $\ind_\A(P)=0$, then $C_\A(P) = I_\A(P)$. Similarly, we have
$C_\D(P) = I_{\D}(P)$ when $\ind_\D(P)=0$. The succeeding lemma follows
immediately from Proposition~\ref{prop:EC}.

\begin{lemma}\label{lem:EC0}
	Let $P\in\P(w_0^{(n+1)})$. 
If $\ind_{\A}(P)=0$, then
\[
P \sim E_\A(C_\A(P), C_\A(P)).
\]
Similarly, if $\ind_{\D}(P)=0$, then
\[
P \sim E_\D(C_\D(P), C_\D(P)).
\]
\end{lemma}

Lemma~\ref{lem:EC0} implies that if $\ind_\A(P)=0$
(respectively,~$\ind_\D(P)=0$), then $P$ is completely determined by $C_\A(P)$
(respectively,~$C_\D(P)$) up to isomorphism.

The following definition will be used frequently throughout this section.

\begin{definition}
  For a word poset $P$, an element $x\in P$ is called a \emph{top element} if
  $x$ is the largest element in its column. In other words, $x\in P$ is a top
  element if $y\le_P x$ for all $y\in P$ with $f_P(y)=f_P(x)$. For $i \in [n]$,
  denote by $m_i(P)$ the top element in column $i$. 
\end{definition}

The following lemma shows that if $P$ is a Gelfand--Cetlin type word poset, the
top elements of $P$ must form a chain. 

\begin{lemma}\label{lem:min}
	Let $P\in\P(w_0^{(n+1)})$. Then 
\[
m_1(P) <_P m_2(P) <_P \cdots <_P m_n(P) \qquad \mbox{if $\ind_{\A}(P)=0$},
\]
\[
m_1(P) >_P m_2(P) >_P \cdots >_P m_n(P) \qquad \mbox{if $\ind_{\D}(P)=0$}.
\]
Moreover, $m_n(P)$ \textup{(}respectively, $m_1(P)$\textup{)} is the maximum element of $P$ if
$\ind_{\A}(P)=0$ \textup{(}respectively, $\ind_{\D}(P)=0$\textup{)}.
\end{lemma}
\begin{proof}
  We will only consider the case $\ind_\A(P)=0$ because the other case
  $\ind_\D(P)=0$ can be proved similarly.

  Since $\ind_{\A}(P)=0$, by Lemma~\ref{lem:EC0}, we have $P \sim E_\A(C_\A(P),
  C_\A(P))$. By definition, $Q:=E_\A(C_\A(P), C_\A(P))$ is the word poset
  obtained from $C_\A(P)$ by adding $n$ elements $a_1,\dots,a_n$ with additional
  covering relations $a_1\lessdot_Q\dots \lessdot_Q a_n$ and $x\lessdot_Q a_i$
  for each maximal element $x$ in $C_\A(P)$ and $i\in[n]$ such that
  $|f_{Q}(x)-f_Q(a_i)|=1$, where $f_Q(x)=f_{C_\A(P)}(x)$ for $x\in C_\A(P)$ and
  $f_Q(a_i)=i$ for $i\in [n]$. By the construction, we have $a_i=m_i(Q)$ for
  $i\in [n]$, and therefore $m_1(Q)<_Q \dots<_Q m_n(Q)$. Since $P\sim Q$, this
  shows the first statement.

  For the second statement, let $x$ be an arbitrary element in $Q$. Suppose
  $f_Q(x)=i$. Since $Q\sim P\in \P(w_0^{(n+1)})$, by Lemma~\ref{lem:chain},
  $\{y\in Q \colon f_Q(y) = i\}$ is a chain in $P$ for each $i\in[n]$. By definition
  $m_i(Q)$ is the maximum element in this chain. Then $x\le_Q m_i(Q)\le_Q
  m_n(Q)$. Thus $x\le_Qm_n(Q)$, and therefore $m_n(Q)$ is the maximum element in
  $Q$. Since $P\sim Q$, this shows the second statement.
\end{proof}

By Lemma~\ref{lem:min}, if $P\in \PGC(n)$ and $n\ge2$, then we have
$\ind_\A(P)=0$ or $\ind_\D(P)=0$, but not both. This means that there is a
unique $\delta=(\delta_1,\dots,\delta_{n-1})\in\{\A,\D\}^{n-1}$ such that
$\ind_{\delta}(P) =(0,\dots,0)$. Therefore the map 
\[
\phi\colon[\PGC(n)]\to \{\A,\D\}^{n-1}
\] 
sending $[P]$ to such $\delta$ is well-defined. This map is in
fact a bijection.

\begin{proposition}\label{prop:phi}
  For $n\ge2$, the map $\phi\colon[\PGC(n)]\to \{\A,\D\}^{n-1}$ is a bijection.
\end{proposition}
\begin{proof}
  We first show that $\phi$ is injective. Suppose that $P\in \PGC(n)$ satisfies
  $\phi([P])=\delta=(\delta_1,\dots,\delta_{n-1})\in\{\A,\D\}^{n-1}$. By
  definition, we have $\ind_{\delta}(P) =(0,\dots,0)$.
  We will show that $[P]$ is determined by $\delta$.

  Define the word posets $P_k\in\PGC(k)$ for $k\in[n]$ recursively as follows.
  First, we set $P_n=P$. For $k\in [n-1]$, define
\[
P_k = C_{\delta_k}(P_{k+1}).
\] 
Since $P_1\in \P(w_0^{(2)})$, $P_1$ is a word poset with one element, say $x$, 
and $f_{P_1}(x)=1$.
By Lemma~\ref{lem:EC0}, for $k\in[n-1]$, we have
\[
  P_{k+1} \sim E_{\delta_k}(C_{\delta_{k}}(P_{k+1}), C_{\delta_{k}}(P_{k+1}))
  =E_{\delta_k}(P_{k}, P_k).
\] 
Thus $P=P_n$ is determined uniquely by $\delta$ up to isomorphism.
This shows that $\phi$ is injective.

To show that $\phi$ is surjective take an arbitrary sequence
$\delta=(\delta_1,\dots,\delta_{n-1})\in\{\A,\D\}^{n-1}$. Define the word posets
$P_k\in\PGC(k)$ for $k\in[n]$ by $P_1\in \P(w_0^{(2)})$ and
\[
  P_{k+1} =E_{\delta_k}(P_{k}, P_k) \qquad \mbox{for $k\in [n-1]$}.
\] 
Here we may choose any $P_1\in \P(w_0^{(2)})$ because if $P_1,P_1'\in
\P(w_0^{(2)})$ then $P_1\sim P_1'$. It is easy to check that $\phi(P_n)=\delta$.
Thus $\phi$ is surjective, which completes the proof.
\end{proof}

The proof of the above proposition shows that if $P,Q\in\P(w_0^{(n+1)})$ satisfy
$\ind_{\delta}(P) = \ind_{\delta}(Q)=(0,\dots,0)$ for some
$\delta\in\{\A,\D\}^{n-1}$, then $P\sim Q$ since both $P$ and $Q$ are determined by
$\delta$. In general, the condition $\ind_{\delta}(P) =
\ind_{\delta}(Q)$ for some $\delta\in\{\A,\D\}^{n-1}$ does not imply
$P\sim Q$ as the following example shows. 

\begin{example}\label{example_total_index}
Consider the two words
\[
\mathbf i =(3,2,1,2,3,4,3,2,3,1), \quad
\mathbf j = (1,3,2,1,4,3,4,2,3,1) \in \RR(w_0^{(5)}).
\]
For any $(\ad_1,\ad_2) \in \{\A,\D\}^3$, we have $\ind_{(\ad_1,\ad_2,\A)}(\Pi)=
\ind_{(\ad_1,\ad_2,\A)}(\Pj)$, but $\Pi\not\sim \Pj$. See
Figure~\ref{fig_example_word_posets} for these word posets. Accordingly, a
single $\delta$-index of $P$ does not always determine the word poset class
$[P]$.

\begin{figure}
	\begin{tabular}{cc}
				\begin{tikzpicture}[auto,node distance=0.1cm and 0.5cm, inner sep=2pt, scale = 0.7]
		\tikzstyle{state}=[draw, circle]
		
		\node[state] (1) at (2,0) {};
		\node[state] (2) at (1,0.5) {};
		\node[state] (3) at (0,1) {};
		\node[state] (4) at (1,1.5) {};
		\node[state] (5) at (2,2) {};
		\node[state] (6) at (3,2.5) {};
		\node[state] (7) at (2,3) {};
		\node[state] (8) at (1,3.5) {};
		\node[state] (9) at (2,4) {};
		\node[state] (10) at (0,4){};

		\draw (1)--(2)--(3)--(4)--(5)--(6)--(7)--(8)--(9)
		(8)--(10);

		\draw[dotted, gray] (3)--(10)--(0,5) node[above] {$1$};
		\draw[dotted, gray] (2)--(4)--(8)--(1, 5) node[above] {$2$};
		\draw[dotted, gray] (1)--(5)--(7)--(9)--(2,5) node[above] {$3$};				
		\draw[dotted, gray] (6)--(3,5) node[above] {$4$};
		\end{tikzpicture} 
		&
		
		\begin{tikzpicture}[auto,node distance=0.1cm and 0.5cm, inner sep=2pt,
		halo/.style={line join=round,
			double,line cap=round,double distance=#1,shorten >=-#1/2,shorten <=-#1/2},
		halo/.default=7mm, 
		scale = 0.7]
		\tikzstyle{state}=[draw, circle]
		
		\node[state] (1) at (0,0) {};
		\node[state] (2) at (2,0) {};
		\node[state] (3) at (1,0.5) {};
		\node[state] (4) at (0,1) {};
		\node[state] (5) at (3,0.5) {};
		\node[state] (6) at (2,1) {};
		\node[state] (7) at (3,1.5) {};
		\node[state] (8) at (1,1.5) {};
		\node[state] (9) at (2,2) {};
		\node[state] (10) at (0,2) {};
		
		\draw (1)--(3)--(4)--(8)--(9)
		(2)--(5)--(6)--(8)--(10)
		(2)--(3)
		(7)--(9)
		(3)--(6)--(7);

		\draw[dotted, gray] (1)--(4)--(10)--(0,3) node[above] {$1$};
		\draw[dotted, gray] (3)--(8)--(1, 3) node[above] {$2$};
		\draw[dotted, gray] (2)--(6)--(9)--(2,3) node[above] {$3$};
		\draw[dotted, gray] (5)--(7)--(3, 3) node[above] {$4$};

		\end{tikzpicture}

		\\[1em]
		$P_{(3,2,1,2,3,4,3,2,3,1)}$. & $P_{(1,3,2,1,4,3,4,2,3,1)}$.
	\end{tabular}
	\caption{The word posets $P_{(3,2,1,2,3,4,3,2,3,1)}$ and $P_{(1,3,2,1,4,3,4,2,3,1)}$.}
	\label{fig_example_word_posets}
\end{figure}
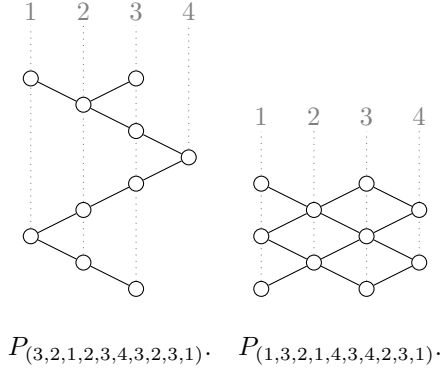

\end{example}

Although a single $\delta$-index of $P$ is not enough to determine the word
poset $[P]$, in the next section, we will show that the $\delta$-indices
$\ind_\delta(P)$ for all $\delta\in\{\A,\D\}^{n-1}$ determine $[P]$ (see
Theorem~\ref{thm_main_injection}). Note that in
Example~\ref{example_total_index} we have $\ind_{\D}(\Pi) = 1 \neq 2 =
\ind_\D(\Pj)$, so $\Pi$ and $\Pj$ do not have the same $\delta$-indices for all
$\delta$.

Proposition~\ref{prop:phi} immediately gives the cardinality of the Gelfand--Cetlin
type commutation classes.

\begin{corollary}\label{cor_main}
  For $n\ge2$, we have
\[
|[\RGC(n)]|=|[\PGC(n)]| = 2^{n-1}.
\]
\end{corollary}

We note that Corollary~\ref{cor_main} was proved in the recent paper
\cite[Proposition~20]{GMS} using a different method.

Using the construction in the proof of Proposition~\ref{prop:phi}  and standard Young tableaux, we can find a
recurrence relation for the number of Gelfand--Cetlin type reduced words. To this end, we need the following lemma, which allows
us to draw the Hasse diagram of a Gelfand--Cetlin type word poset $P$
corresponding to $\delta\in\{\A,\D\}^{n-1}$ using only $\delta$.

\begin{lemma}\label{lem:con}
  Let $P\in\P(w_0^{(n+1)})$ and let $\delta=(\delta_1,\dots,\delta_{n-1})$ be
  the element in $\{\A,\D\}^{n-1}$ satisfying $\ind_\delta(P)=(0,\dots,0)$. Let
  $P_n=P$, and for $k\in [n-1]$, define
\[
P_k = C_{\delta_k}(P_{k+1}).
\] 
Then $P_1$ is isomorphic to the unique word poset \textup{(}up to isomorphism\textup{)} in
$\P(w_0^{(2)})$ 
and, for $k\in[n-1]$, the word poset $P_{k+1}$ is constructed as follows
\textup{(}see Figure~\ref{fig:P_k}\textup{)}.
\begin{itemize}
\item If $\delta_{k-1}=\A$ and $\delta_{k}=\A$, then the Hasse diagram of
  $P_{k+1}$ is obtained from that of $P_{k}$ by adding the chain
  $m_1(P_{k+1})\lessdot_{P_{k+1}} \dots \lessdot_{P_{k+1}} m_{k+1}(P_{k+1})$ with additional
  covering relations $m_{i}(P_{k+1})\gtrdot_{P_{k+1}} m_i(P_{k})$ for $i\in
  [k]$.
\item If $\delta_{k-1}=\A$ and $\delta_{k}=\D$, then the Hasse diagram of
  $P_{k+1}$ is obtained from that of $P_{k}$ by adding the chain
  $m_1(P_{k+1})\gtrdot_{P_{k+1}} \dots \gtrdot_{P_{k+1}} m_{k+1}(P_{k+1})$ with an additional
  covering relation $m_{k}(P_{k})\lessdot_{P_{k+1}} m_{k+1}(P_{k+1})$.
\item If $\delta_{k-1}=\D$ and $\delta_{k}=\A$, then the Hasse diagram of
  $P_{k+1}$ is obtained from that of $P_{k}$ by adding the chain
  $m_1(P_{k+1})\lessdot_{P_{k+1}} \dots \lessdot_{P_{k+1}} m_{k+1}(P_{k+1})$ with an additional
  covering relation $m_{1}(P_{k+1})\gtrdot_{P_{k+1}} m_1(P_{k})$.
\item If $\delta_{k-1}=\D$ and $\delta_{k}=\D$, then the Hasse diagram of
  $P_{k+1}$ is obtained from that of $P_{k}$ by adding the chain
  $m_1(P_{k+1})\gtrdot_{P_{k+1}} \dots \gtrdot_{P_{k+1}} m_{k+1}(P_{k+1})$ with additional
  covering relations $m_{i}(P_{k})\lessdot_{P_{k+1}} m_{i+1}(P_{k+1})$ for $i\in
  [k]$.
\end{itemize}
\end{lemma}
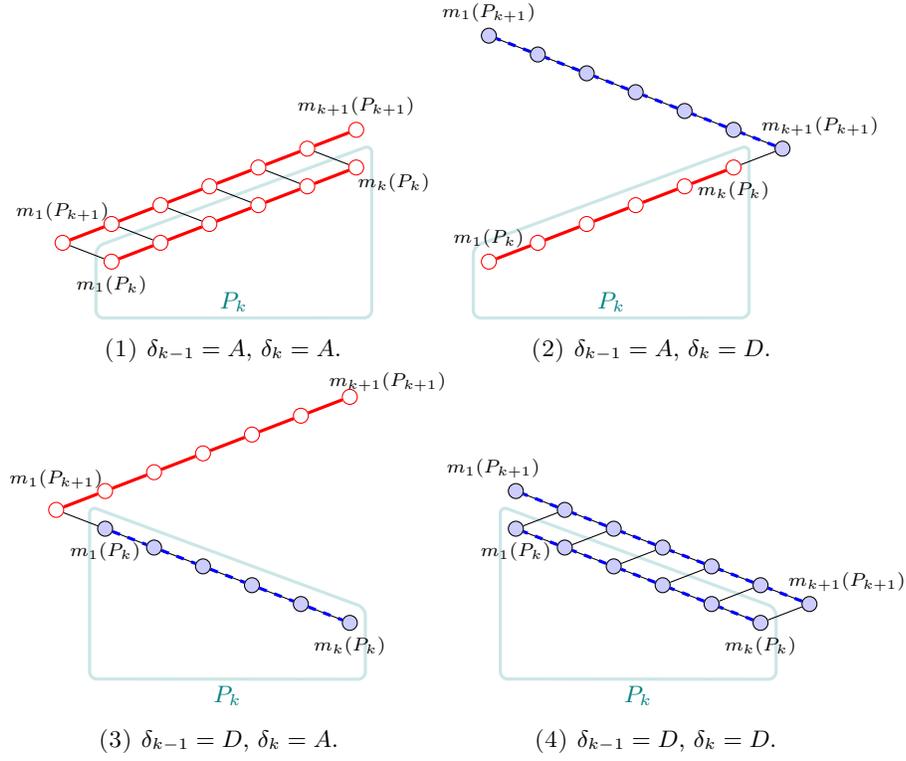
\begin{figure}
	\centering
	\begin{subfigure}[b]{0.45\textwidth}
		\centering
		\begin{tikzpicture}[auto,node distance=0.1cm and 0.5cm, inner sep=2pt]
		\tikzstyle{A}=[draw, circle,  draw=red]
		\tikzstyle{D}=[draw, circle,  fill=blue!20]
		
		\node[A] (1) {};
		\node[A] (2) [above right = of 1] {};
		\node[A] (3) [above right = of 2] {};
		\node[A] (4) [above right = of 3] {};
		\node[A] (5) [above right = of 4] {};
		\node[A] (6) [above right = of 5] {};
		\node[A] (7) [above right = of 6] {};
		
		\node[A] (8) [below right = of 1] {};
		\node[A] (9) [below right = of 2] {};
		\node[A] (10) [below right = of 3] {};
		\node[A] (11) [below right = of 4] {};
		\node[A] (12) [below right = of 5] {};
		\node[A] (13) [below right = of 6] {};
		
		\draw[rounded corners=1mm, draw=teal!20!white, very thick] 
		($(13.east) + (0.1,-2)$)--($(13.east) + (0.1,0.3)$) -- ($(8.west)+(-0.1,0.2)$)--($(8.west)+(-0.1,-0.75)$) --cycle  node[midway] {\textcolor{teal}{\small $P_k$}};
		
		\draw (1)--(2)--(3)--(4)--(5)--(6)--(7)
		(8)--(9)--(10)--(11)--(12)--(13)
		(1)--(8)
		(2)--(9)
		(3)--(10)
		(4)--(11)
		(5)--(12)
		(6)--(13);
		
		    \draw[very thick, red](1)--(2)--(3)--(4)--(5)--(6)--(7)
		    (8)--(9)--(10)--(11)--(12)--(13);

		\node [ above  = of 1] {\scriptsize $m_1(P_{k+1})$};
		\node at ($(13) + (0.5,-0.2)$) {\scriptsize $m_{k}(P_{k})$};
		\node [below = 0cm of 8] {\scriptsize $m_{1}(P_{k})$};
		\node [above  = 0cm of 7] {\scriptsize $m_{k+1}(P_{k+1})$};
		
		\end{tikzpicture}
		\caption{$\delta_{k-1}= A$, $\delta_{k} = A$.}
	\end{subfigure}%
	\begin{subfigure}[b]{0.45\textwidth}
		\centering
		\begin{tikzpicture}[auto,node distance=0.1cm and 0.5cm, inner sep=2pt]
		\tikzstyle{A}=[draw, circle,  draw=red]
		\tikzstyle{D}=[draw, circle,  fill=blue!20]

		\node[D] (1) {};
		\node[D] (2) [above left = of 1] {};
		\node[D] (3) [above left = of 2] {};
		\node[D] (4) [above left = of 3] {};
		\node[D] (5) [above left = of 4] {};  	
		\node[D] (6) [above left = of 5] {};
		\node[D] (7) [above left = of 6] {};
		
		\node[A] (8) [below left = of 1] {};
		\node[A] (9) [below left = of 8] {};
		\node[A] (10) [below left = of 9] {}; 	
		\node[A] (11) [below left = of 10] {};
		\node[A] (12) [below left = of 11] {};
		\node[A] (13) [below left = of 12] {};
		
		\draw[rounded corners=1mm, draw=teal!20!white, very thick] 
($(8.east) + (0.1,-2)$)--($(8.east) + (0.1,0.3)$) -- ($(13.west)+(-0.1,0.2)$)--($(13.west)+(-0.1,-0.75)$) --cycle  node[midway] {\textcolor{teal}{\small $P_k$}};

		\draw (1)--(2)--(3)--(4)--(5)--(6)--(7)
		(1)--(8)--(9)--(10)--(11)--(12)--(13);

		 \draw[very thick, dashed, blue] (1)--(2)--(3)--(4)--(5)--(6)--(7);
		 \draw[very thick, red](8)--(9)--(10)--(11)--(12)--(13);
		
		\node   at ($(1) + (0.5,0.3)$){\scriptsize $m_{k+1}(P_{k+1})$};
		\node [ above = 0cm of 7] {\scriptsize $m_{1}(P_{k+1})$};
		\node [below = 0cm of 8] {\scriptsize $m_{k}(P_{k})$};
		\node [above  = -0cm of 13] {\scriptsize $m_{1}(P_{k})$};
		
		\end{tikzpicture}
		\caption{$\delta_{k-1} = A$, $\delta_{k} = D$.}
	\end{subfigure}
	\begin{subfigure}[b]{0.45\textwidth}
		\centering
		\begin{tikzpicture}[auto,node distance=0.1cm and 0.5cm, inner sep=2pt]
		\tikzstyle{A}=[draw, circle,  draw=red]
		\tikzstyle{D}=[draw, circle,  fill=blue!20]

		\node[A] (1) {};
		\node[A] (2) [above right = of 1] {};
		\node[A] (3) [above right = of 2] {};
		\node[A] (4) [above right = of 3] {};
		\node[A] (5) [above right = of 4] {};  	
		\node[A] (6) [above right = of 5] {};
		\node[A] (7) [above right = of 6] {};
		
		\node[D] (8) [below right = of 1] {};
		\node[D] (9) [below right = of 8] {};
		\node[D] (10) [below right = of 9] {}; 	
		\node[D] (11) [below right = of 10] {};
		\node[D] (12) [below right = of 11] {};
		\node[D] (13) [below right = of 12] {};
		
		\draw[rounded corners=1mm, draw=teal!20!white, very thick] 
($(8.west) + (-0.1,-2)$)--($(8.west) + (-0.1,0.3)$) -- ($(13.east)+(0.1,0.2)$)--($(13.east)+(0.1,-0.75)$) --cycle  node[midway] {\textcolor{teal}{\small $P_k$}};

		\draw (1)--(2)--(3)--(4)--(5)--(6)--(7)
		(1)--(8)--(9)--(10)--(11)--(12)--(13);
		
 \draw[very thick, dashed, blue]		(8)--(9)--(10)--(11)--(12)--(13);
 \draw[very thick, red] (1)--(2)--(3)--(4)--(5)--(6)--(7);
		\node [ above  = of 1] {\scriptsize $m_{1}(P_{k+1})$};
		\node at ($(7) + (0.5,0.2)$) {\scriptsize $m_{k+1}(P_{k+1})$};
		\node [below = 0cm of 8] {\scriptsize $m_{1}(P_{k})$};
		\node [below  = -0cm of 13] {\scriptsize $m_{k}(P_{k})$};
		
		\end{tikzpicture}
		\caption{$\delta_{k-1} = D$, $\delta_{k} = A$.}
	\end{subfigure}
	\begin{subfigure}[b]{0.45\textwidth}
		\centering
		\begin{tikzpicture}[auto,node distance=0.1cm and 0.5cm, inner sep=2pt]
		\tikzstyle{A}=[draw, circle,  draw=red]
		\tikzstyle{D}=[draw, circle,  fill=blue!20]
		
		\node[D] (1) {};
		\node[D] (2) [above left = of 1] {};
		\node[D] (3) [above left = of 2] {};
		\node[D] (4) [above left = of 3] {};
		\node[D] (5) [above left = of 4] {};
		\node[D] (6) [above left = of 5] {};
		\node[D] (7) [above left = of 6] {};
		
		\node[D] (8) [below left = of 1] {};
		\node[D] (9) [below left = of 2] {};
		\node[D] (10) [below left = of 3] {};
		\node[D] (11) [below left = of 4] {};
		\node[D] (12) [below left = of 5] {};
		\node[D] (13) [below left = of 6] {};
		
		\draw[rounded corners=1mm, draw=teal!20!white, very thick] 
($(13.west) + (-0.1,-2)$)--($(13.west) + (-0.1,0.3)$) -- ($(8.east)+(0.1,0.2)$)--($(8.east)+(0.1,-0.75)$) --cycle  node[midway] {\textcolor{teal}{\small $P_k$}};
		
		\draw (1)--(2)--(3)--(4)--(5)--(6)--(7)
		(8)--(9)--(10)--(11)--(12)--(13)
		(1)--(8)
		(2)--(9)
		(3)--(10)
		(4)--(11)
		(5)--(12)
		(6)--(13);

		 \draw[very thick, dashed, blue] (1)--(2)--(3)--(4)--(5)--(6)--(7)
		 		(8)--(9)--(10)--(11)--(12)--(13);
		
		\node  at ($(1) + (0.5,0.3)$) {\scriptsize $m_{k+1}(P_{k+1})$};
		\node at ($(7) + (-0.3,0.3)$) {\scriptsize $m_{1}(P_{k+1})$};
		\node [below = 0cm of 8] {\scriptsize $m_{k}(P_{k})$};
		\node [below  = -0cm of 13] {\scriptsize $m_{1}(P_{k})$};
		
		\end{tikzpicture}
		\caption{$\delta_{k-1} = D$, $\delta_{k} = D$.}
	\end{subfigure}%
	\caption{The Hasse diagrams of $P_k$ and $P_{k+1}$.}
	\label{fig:P_k}
\end{figure}
\begin{proof}
  Consider the case that $\delta_{k-1}=\A$ and $\delta_{k}=\A$. Since $P_k =
  C_{\delta_k}(P_{k+1}) = C_{\A}(P_{k+1})$ and $\ind_\A(P_{k+1})=0$, by
  Lemma~\ref{lem:EC0}, we have
\[
P_{k+1} \sim E_\A(P_k, P_k).
\] 
Then it is straightforward to check that we obtain the desired description for
$P_{k+1}$ by the definition of $\A$-extension in
Definition~\ref{def:extensions}. The other three cases can be checked similarly.
\end{proof}

Now we define standard Young tableaux of shifted shape.

\begin{definition}
  A \emph{partition} of $n$ is a weakly decreasing sequence $\mu =
  (\mu_1,\dots,\mu_t)$ of positive integers summing to $n$. A partition
  $\mu = (\mu_1,\dots,\mu_t)$ is \emph{strict} if
  $\mu_1>\dots>\mu_t$. For a strict partition $\mu =
  (\mu_1,\dots,\mu_t)$, the \emph{shifted diagram} of $\mu$, denoted
  $\mu^\ast$, is the set
	\[
	\mu^{\ast} \coloneqq \{ (i,j) \mid 1 \leq i \leq t, i \leq j \leq \mu_i + i - 1\}.
	\]
  We will identify $\mu^\ast$ as an array of squares where there is a
  square in row $i$ and column $j$ for each $(i,j)\in\mu^\ast$. For a strict
  partition $\mu = (\mu_1,\dots,\mu_t)$ of $n$, a \emph{standard
    Young tableau} of shifted shape $\mu$ is a bijection $T \colon \mu^*\to
  [n]$ such that $T(i,j)\le T(i',j')$ if $i\le i'$ and $j\le j'$. We will
  represent a standard Young tableau $T$ of shifted shape $\mu$ by filling
  $T(i,j)$ in the square in row $i$ and column $j$ of $\mu^\ast$.
  Denote by $g^\mu$ the number of standard Young tableaux of shifted shape $\mu$.
\end{definition}

For example, the shifted diagram of shape $\mu = (3,2,1)$ is drawn as follows.
\[
\ytableausetup{smalltableaux}
\ytableausetup{aligntableaux = center}
\mu^{\ast} = \ydiagram{3, 1+2, 2+1}
\]
There are two standard Young tableaux of shifted shape $\mu=(3,2,1)$:
\[
\ytableausetup{notabloids}
\begin{ytableau}
1 & 2 & 3 \\
\none & 4 & 5 \\
\none & \none & 6 
\end{ytableau} 
\quad
\begin{ytableau}
1 & 2 & 4 \\
\none & 3 & 5 \\
\none & \none & 6 
\end{ytableau}
\]

Thrall~\cite{Thrall} showed that the number $g^{\mu}$ of standard Young
tableaux of shifted shape $\mu = (\mu_1,\dots,\mu_t)$ is given as follows:
\begin{equation}\label{eq_g_lambda}
g^{\mu} = \frac{|\mu|!}{\mu_1 ! \mu_2 ! \cdots \mu_t !}
\prod_{i < j} \frac{\mu_i - \mu_j}{\mu_i +\mu_j}.
\end{equation}
For instance, if $\mu = (3,2,1)$, then we have that
\[
g^{(3,2,1)} = \frac{6!}{3!2!1!} \cdot \frac{1 \cdot 2 \cdot 1}{5 \cdot 4 \cdot 3} = 2.
\]
There is another formula for $g^{\mu}$ called the (shifted) hook length formula,
see \cite[p.267, eq.(2)]{Macdonald}.

For a strict partition $\mu = (\mu_1,\dots,\mu_t)$, we define
$Q_{\mu}$ to be the poset on $\mu^\ast$ with relations $(i,j)
\le_{Q_\mu} (i',j')$ if $i\le i'$ and $j\le j'$. For example, if $\mu =
(4,3,2)$, then the poset $Q_{\mu}$ is given as follows.
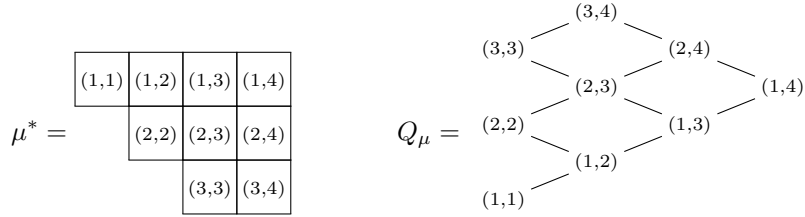
\[
\mu^{\ast} = 
\ytableausetup{nosmalltableaux,boxsize = 2em}
\begin{ytableau}
{\scriptstyle (1,1)} & {\scriptstyle (1,2)} & {\scriptstyle (1,3)}  & {\scriptstyle (1,4)} \\
\none & {\scriptstyle (2,2)} & {\scriptstyle (2,3)} & {\scriptstyle (2,4)}  \\
\none & \none  & {\scriptstyle (3,3)} & {\scriptstyle (3,4)} 
\end{ytableau}
\qquad \qquad
Q_{\mu} = 
\raisebox{-1cm}{
\begin{tikzpicture}[auto,node distance=0.1cm and 0.5cm, inner sep=2pt]

\node (1) {${\scriptstyle (3,4)}$};
\node (2) [below left = of 1] {${\scriptstyle (3,3)}$};
\node (3) [below right = of 1] {${\scriptstyle (2,4)}$};
\node (4) [below left = of 3] {${\scriptstyle (2,3)}$};
\node (5) [below left = of 4] {${\scriptstyle (2,2)}$};
\node (6) [below right = of 3] {${\scriptstyle (1,4)}$};
\node (7) [below left = of 6] {${\scriptstyle (1,3)}$};
\node (8) [below left = of 7] {${\scriptstyle (1,2)}$};
\node (9) [below left = of 8] {${\scriptstyle (1,1)}$};

\draw (1)--(2)
(3)--(4)--(5)
(6)--(7)--(8)--(9)
(1)--(3)--(6)
(2)--(4)--(7)
(5)--(8);
%
%
%

\end{tikzpicture} }
\]

There is a natural bijection between the standard Young tableaux of shifted
shape $\mu$ and the linear extensions of $Q_\mu$. Therefore the number
of linear extensions of $Q_\mu$ is equal to $g^\mu$.

Now we are ready to state our first main theorem.

\begin{theorem}\label{thm_SYT}
	Let $\gc(n)$ be the number of Gelfand--Cetlin type reduced words
  in~$\Rn{n+1}$, i.e., $\gc(n)=|\RGC(n)|$. Then $\gc(0)=\gc(1)=1$ and for $n\ge2$, we have
	\[
	\gc(n) = \sum_{i=1}^n g^{(n,n-1,\dots,n-i+1)} \gc(n-i),
	\]
	where $g^{\mu}$ is the number of standard Young tableaux of shifted shape $\mu$ \textup{(}see~\eqref{eq_g_lambda}\textup{)}.
\end{theorem}
\begin{proof}
  Clearly, we have $\gc(0)=\gc(1)=1$. Suppose $n\ge2$. 
  Observe that
  \begin{equation}
    \label{eq:2}
\gc(n)=|\RGC(n)| = \sum_{[\mathbf{i}]\in [\RGC(n)]} |[\mathbf{i}]|.
  \end{equation}
  By Proposition~\ref{prop:RGC=PGC}, the map $[\mathbf{i}]\mapsto [\Pi]$ is a
  bijection from $[\RGC(n)]$ to $[\PGC(n)]$. Moreover, by
  Proposition~\ref{prop_poset_determines_commutations_class}, $|[\mathbf{i}]|$
  is equal to the number of linear extensions of $\Pi$. This shows that we can
  rewrite \eqref{eq:2} as
\begin{equation}
  \label{eq:3}
\gc(n)=\sum_{[P]\in [\PGC(n)]} e(P),
\end{equation}
where $e(P)$ is the number of linear extensions of $P$.

Define $\an(n)$ and $\dn(n)$ by
\[
\an(n)=\sum_{\substack{[P]\in [\PGC(n)],\\ \ind_\A(P)=0}} e(P), \qquad
\dn(n)=\sum_{\substack{[P]\in [\PGC(n)],\\ \ind_\D(P)=0}} e(P).
\]
We claim that
\begin{align}
  \label{eq:an}
	\an(n) &= \sum_{i=1}^n g^{(n,n-1,\dots,n-i+1)} \dn(n-i),\\
  \label{eq:dn}
	\dn(n) &= \sum_{i=1}^n g^{(n,n-1,\dots,n-i+1)} \an(n-i).
\end{align}
Since $\gc(n)=\an(n)+\dn(n)$, the identity in this theorem is obtained by adding
\eqref{eq:an} and \eqref{eq:dn}. Thus it suffices to show these two identities.
We will only show \eqref{eq:an} because \eqref{eq:dn} can be shown similarly.

To show \eqref{eq:an}, consider $[P]\in[\PGC(n)]$ with $\ind_\A(P)=0$. By
Proposition~\ref{prop:phi}, there is a unique $\delta\in \{\A,\D\}^{n-1}$ such
that $\ind_{\delta}(P) = (0,\dots,0)$. Let $P_n=P$, and for $k\in [n-1]$, define
\[
P_k = C_{\delta_k}(P_{k+1}).
\] 
Then, by Lemma~\ref{lem:con}, $P_{k+1}$ is obtained from $P_k$ by adding a
descending or ascending chain of length $k+1$ depending on $\delta_{k-1}$ and
$\delta_{k}$. Since $\ind_\A(P)=0$, there is a unique integer $i\in[n-1]$ such
that $\delta_{n-1}=\delta_{n-2} = \dots=\delta_{n-i}=\A$ and $\delta_{n-i-1}=
\D$, where the second condition is ignored if $i=n-1$. By Lemma~\ref{lem:con},
one can easily see that $P=P_n$ is obtained from $P_{n-i}$ by adding the poset
$Q_{(n,n-1,\dots,n-i+1)}$ above it as shown in Figure~\ref{fig:syt}. Since every
element of $P_{n-i}$ is smaller than every element of $Q_{(n,n-1,\dots,n-i+1)}$,
we have
\[
e(P)=e(P_{n-i}) e(Q_{(n,n-1,\dots,n-i+1)}) = e(P_{n-i})g^{(n,n-1,\dots,n-i+1)}.
\]
Note that $[P_{n-i}]\in [\PGC(n-i)]$ and $\ind_D(P_{n-i})=0$. Conversely, for
any such $P_{n-i}$, one can construct $P$ in this way. This shows that
\[
	\an(n) = \sum_{i=1}^n g^{(n,n-1,\dots,n-i+1)} 
\sum_{\substack{[P_{n-i}]\in[\PGC(n-i)],\\ \ind_\D(P)=0}} e(P_{n-i}),
 \]
 which is the same as~\eqref{eq:an}. Similarly, we obtain the formula~\eqref{eq:dn} and the
 proof is completed.
  \end{proof}

\begin{figure}
    \centering
    \begin{tikzpicture}[auto,node distance=0.1cm and 0.5cm, inner sep=2pt]
    \tikzstyle{state1}=[draw, circle,  draw=red]
    \tikzstyle{state2}=[draw, circle,  fill=blue!20]
    
    \node[draw, circle] (1) {};
    \node[state1] (2) [above left = of 1] {};
    \node[state1] (3) [above right = of 2] {};
    \node[state1] (4) [above left = of 2] {};
    \node[state1] (5) [above right = of 4] {};
    \node[state1] (6) [above right = of 5] {};
    \node[state2] (7) [above right = of 6] {};
    \node[state2] (8) [above left = of 7] {};
    \node[state2] (9) [above left = of 8] {};
    \node[state2] (10) [above left = of 9] {};
    \node[state2] (11) [above right = of 7] {};
    \node[state2] (12) [above left = of 11] {};
    \node[state2] (13) [above left = of 12] {};
    \node[state2] (14) [above left = of 13] {};
    \node[state2] (15) [above left = of 14] {};
    \node[state1] (16) [above left = of 15] {};
    \node[state1] (17) [above right = of 16] {};
    \node[state1] (18) [above right = of 17] {};
    \node[state1] (19) [above right = of 18] {};
    \node[state1] (20) [above right = of 19] {};
    \node[state1] (21) [above right = of 20] {};
    \node[state1] (22) [above left = of 16] {};
    \node[state1] (23) [above right = of 22] {};
    \node[state1] (24) [above right = of 23] {};
    \node[state1] (25) [above right = of 24] {};
    \node[state1] (26) [above right = of 25] {};
    \node[state1] (27) [above right = of 26] {};
    \node[state1] (28) [above right = of 27] {};
    \node[state1] (29) [above left = of 22] {};
    \node[state1] (30) [above right = of 29] {};
    \node[state1] (31) [above right = of 30] {};
    \node[state1] (32) [above right = of 31] {};
    \node[state1] (33) [above right = of 32] {};
    \node[state1] (34) [above right = of 33] {};
    \node[state1] (35) [above right = of 34] {};
    \node[state1] (36) [above right = of 35] {};

\draw (1)--(2)--(4)
(3)--(5)--(6)--(7)--(8)--(9)--(10)
(11)--(12)--(13)--(14)--(15)--(16)--(17)--(18)--(19)--(20)--(21)
(22)--(23)--(24)--(25)--(26)--(27)--(28)
(29)--(30)--(31)--(32)--(33)--(34)--(35)--(36)
(2)--(3)
(4)--(5)
(7)--(11)
(8)--(12)
(9)--(13)
(10)--(14)
(16)--(22)--(29)
(17)--(23)--(30)
(18)--(24)--(31)
(19)--(25)--(32)
(20)--(26)--(33)
(21)--(27)--(34)
(28)--(35);

\draw [decorate,decoration={brace,amplitude=10pt},xshift=4pt,yshift=0pt] (1.5,2.5) -- (1.5,0) node [black,midway,xshift=0.6cm] {$P_{n-i}$};
\draw [decorate,decoration={brace,amplitude=10pt},xshift=4pt,yshift=0pt] (1.5,5) -- (1.5,2.75) node [black,midway,xshift=10pt] {$Q_{(n,n-1,\dots,n-i+1)}$};

    \draw[very thick, dashed, blue]  (7) --(8) node[below right = 0.2cm, near start] {$\D_4$}
    (8)--(9)--(10);
    \draw[very thick, dashed, blue]  (11)--(12) node[below right =  0.2cm, near start] {$\D_5$}
    (12)--(13)--(14)--(15);
    
    \draw [very thick, red]  (2)--(3) node[below left = 0.2cm, near start] {$\A_2$};
    \draw[very thick, red]  (4)--(5) node[below left = 0.2cm, near start] {$\A_3$}
    (5)--(6);
    \draw[very thick, red]  (16)--(17) node[below left = 0.2cm, near start] {$\A_6$}
    (17)--(18)--(19)--(20)--(21);
    \draw[very thick, red] (22)--(23) node[below left = 0.2cm, near start] {$\A_7$}
    (23)--(24)--(25)--(26)--(27)--(28);
    \draw[very thick, red] (29)--(30) node[below left = 0.2cm, near start] {$\A_8$}
    (30)--(31)--(32)--(33)--(34)--(35)--(36);
    
    \end{tikzpicture}
    \caption{The Hasse diagram of a word poset $P\in\PGC(n)$ such that
      $\ind_\delta(P)=(0,\dots,0)$ for $\delta=(\A,\A,\D,\D,\A,\A,\A)$. In
      this case, $\delta_{n-1}=\delta_{n-2}=\dots=\delta_{n-i}=\A$ and
      $\delta_{n-i}=\D$, where $n=8$ and $i=3$. For $r=2,3,\dots,n$, the set
      $P_r\setminus P_{r-1}$ forms is an ascending chain $\A_r$ or a
      descending chain $\D_r$. The word poset $P=P_n$ is
      decomposed into two parts $P_{n-i}$ and $P_n\setminus P_{n-i}\sim
      Q_{(n,n-1,\dots,n-i+1)}$.}
    \label{fig:syt}
  \end{figure}
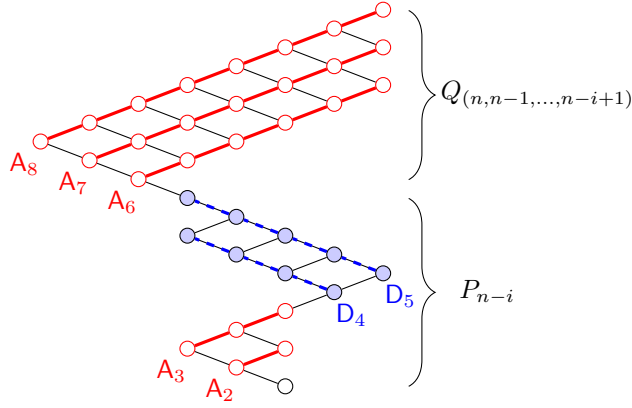

By \eqref{eq_g_lambda}, we have
\[
\begin{split}
&g^{(2,1)} = 1, \quad g^{(2)} = 1, \\
&g^{(3,2,1)} = 2, \quad g^{(3,2)} = 2, \quad g^{(3)} = 1, \\
&g^{(4,3,2,1)} = 12, \quad g^{(4,3,2)} = 12, \quad  g^{(4,3)} = 5,\quad  g^{(4)} = 1.
\end{split}
\]
Applying Theorem~\ref{thm_SYT}, we can compute $\gc(n)$ for $n=2,3,4$ as follows. 
\[
\begin{split}
\gc(2) &=g^{(2,1)} \gc(0) + g^{(2)} \gc(1) = 1+1 = 2, \\
\gc(3) &= g^{(3,2,1)}\gc(0) + g^{(3,2)}\gc(1) + g^{(3)}\gc(2) = 2 +2 +2 = 6,\\
\gc(4) &= g^{(4,3,2,1)}\gc(0) + g^{(4,3,2)}\gc(1) + g^{(4,3)}\gc(2) + g^{(4)}\gc(3) 
= 12 + 12 + 10 + 6 = 40.
\end{split}
\]

We present the first few terms of $\gc(n)$ in Table~\ref{table_gcn}.
\begin{table}[h]
	\begin{tabular}{c|ccccccccc}
		\toprule
		$n$&$0$&$1$&$2$&$3$&$4$&$5$&$6$&$7$ & $8$ \\
		\midrule
		$\gc(n)$& $1$ & $1$ & $2$ & $6$ & $40$ & $916$ & $102176$ &  $68464624$ & $317175051664$ \\
		\bottomrule
	\end{tabular}		
	\caption{The first few terms of $\gc(n)$.}\label{table_gcn}
\end{table}

We close  this section by presenting the
following corollary of Theorem~\ref{thm_SYT}.
\begin{corollary}\label{cor_main2}
  Let $\lambda$ be a regular dominant weight of $\SL_{n+1}(\C)$. The number of reduced words $\mathbf i \in
  \RR(w_0^{(n+1)})$ such that the string polytope $\Delta_{\mathbf i}(\lambda)$
  is unimodularly equivalent to the Gelfand--Cetlin polytope $\GC(\lambda)$ is
  the same as $\gc(n)$.
\end{corollary}
\begin{proof}
	We first  recall the known result from~\cite[Theorem~A]{CKLP} that for
  $\mathbf i \in \RR(w_0^{(n+1)})$, the string polytope $\Delta_{\mathbf
    i}(\lambda)$ is unimodularly equivalent to the Gelfand--Cetlin polytope
  $\GC(\lambda)$ if and only if the string polytope $\Delta_{\mathbf
    i}(\lambda)$ has exactly $n(n+1)$ facets. Here,
  facets are codimension one faces. We note that the number of facets of any
  full dimensional Gelfand--Cetlin polytope of rank $n$ is $n(n+1)$ (cf.~\cite{ACK18}).
  Accordingly, if the string polytope $\Delta_{\mathbf i}(\lambda)$ is combinatorially
  equivalent to a full dimensional Gelfand--Cetlin polytope of rank
    $n$, then it is also unimodularly equivalent to $\GC(\lambda)$ because it
  has $n(n+1)$ facets. This proves the corollary.
\end{proof}

\section{Word posets are determined by $\delta$-indices}
\label{secInjectivityOfTheTotalIndexMap}

In this section, we prove that the $\delta$-indices of $P$ for all $\ad \in
\{\A,\D\}^{n-1}$ completely determine $P\in \P(w_0^{(n+1)})$ up to isomorphism.
Equivalently, the $\delta$-indices of $\mathbf{i}\in\Rn{n+1}$ for all $\ad \in
\{\A,\D\}^{n-1}$ determine the commutation class $[\mathbf{i}]\in
[\RR(w_0^{(n+1)})]$.

\begin{theorem}\label{thm_main_injection}
  Let $P,Q \in \P(w_0^{(n+1)})$. If $\ind_\ad(P) = \ind_\ad(Q)$ for all $\ad
  \in \{\A,\D\}^{n-1}$, then $P\sim Q$.
\end{theorem}

In order to prove Theorem~\ref{thm_main_injection}, we need the following two
lemmas.

\begin{lemma}\label{lemma_An_An-1_in_PD}
  Let $P\in \P(w_0^{(n+1)})$. Then
  \begin{align*}
\A(P)\cap C_\D(P) &= \A(C_\D(P)),\\
\D(P)\cap C_\A(P) &= \D(C_\A(P)).
  \end{align*}
  In other words, the ascending chain of $P$ restricted to $C_\D(P)$ is the
  ascending chain of $C_\D(P)$, and similarly, the descending chain of $P$
  restricted $C_\A(P)$ is the descending chain of $C_\A(P)$.
\end{lemma}

\begin{proof}
By definition we can write
  \begin{align*}
\D(P) &= \{d_1<_P d_{2}<_P \dots <_P d_n\},\\
\A(P) &= \{a_1<_P a_{2}<_P \dots <_P a_n\},
  \end{align*}
  where $f_P(a_i)=i$ and $f_P(d_i)=n+1-i$ for $1\le i\le n$. By
  Proposition~\ref{prop_index}, $\D(P)\cap\A(P)$ has a unique element, say
  $a_k$. By definition, $C_\D(P)=P\setminus \D$ and
\begin{equation}\label{eq:a_k}
 \{a_1<_{C_\D(P)} \dots <_{C_\D(P)} a_{k-1}<_{C_\D(P)} a_{k+1} <_{C_\D(P)} \dots <_{C_\D(P)} a_n\}.
\end{equation}
Moreover, since $a_{1},\dots,a_{k-1}$ are below the descending chain $\D(P)$ and
$a_{k+1},\dots,a_n$ are above $\D(P)$ in the Hasse diagram of $P$, we have
$f_{C_\D(P)}(a_i)=i$ for $1\le i\le k-1$ and $f_{C_\D(P)}(a_i)=i-1$ for $k+1\le
i\le n$, see Figure~\ref{fig_An_An-1_PD}. Therefore \eqref{eq:a_k} is the
ascending chain of $C_\D(P)$, which shows the first identity. The second
identity can be proved similarly.
\end{proof}

	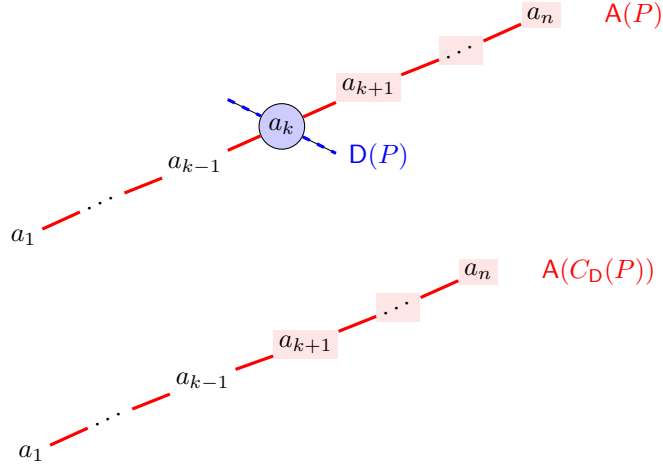
\begin{figure}
	\begin{tikzpicture}[auto,node distance=0.1cm and 0.5cm, inner sep=2pt]
	\tikzstyle{state}=[draw, circle]
	
	\node (1) {$a_1$};
	\node (2) [above right =of 1]  {\rotatebox{26}{$\dots$}};
	\node (k1) [above right = of 2] {$a_{k-1}$};
	\node[state, fill=blue!20] (k) [above right = of k1] {$a_k$};
	\node[fill=red!10] (k2) [above right = of k] {$a_{k+1}$};
	\node[fill=red!10] (d) [above right = of k2] {\rotatebox{26}{$\dots$}};
	\node [fill=red!10] (n) [above right = of d] {$a_n$};
	
	\node (k11) [above left = of k] {};
	\node (k21) [below right = of k] {};
	
	\draw[very thick, red] (1)--(2)--(k1)--(k)--(k2)--(d)--(n) ;

	\node [red, right = of n] {$\A(P)$};
	
	\draw (k11)--(k)--(k21);
	\draw[blue, very thick, dashed] (k11)--(k)--(k21) node[right] {$\D(P)$};
	\end{tikzpicture} \hspace{1cm}%
		\begin{tikzpicture}[auto,node distance=0.1cm and 0.5cm, inner sep=2pt]
	\tikzstyle{state}=[draw, circle]
	
	\node (1) {$a_1$};
	\node (2) [above right =of 1]  {\rotatebox{26}{$\dots$}};
	\node (k1) [above right = of 2] {$a_{k-1}$};
	\node[fill=red!10] (k) [above right = of k1] {$a_{k+1}$};
	\node[fill=red!10] (d) [above right = of k] {\rotatebox{26}{$\dots$}};
	\node[fill=red!10] (n) [above right = of d] {$a_{n}$};
	\node [red, right = of n] {$\A(C_\D(P))$};
	
	\draw[very thick, red] (1)--(2)--(k1)--(k)--(d)--(n);
	
	\end{tikzpicture}
	\caption{The ascending chain $\A(P)$ induces the ascending chain $A(C_\D(P))$.}
	\label{fig_An_An-1_PD}
	\end{figure}

  The following lemma shows that an ideal of a word poset $P\in\P(w_0^{(n+1)})$
  is determined by the number of elements in each column.

\begin{lemma}\label{lemma_order_ideal_n_of_elements}
	Let $P\in\P(w_0^{(n+1)})$. Suppose that $I$ and $J$ are ideals of $P$ such that
\[
|\{ x\in I: f_P(x) = i\}|  = |\{ x\in J: f_P(x) = i\}|,
\]
for all $1\le i\le n$. Then we have $I = J$.
\end{lemma}
\begin{proof}
  Consider the ideals $I$ and $J$ in the statement of this lemma. Observe that $I$ is the
  disjoint union of $\{ x\in I: f_P(x) = i\}$ for $1\le i\le n$. Thus it
  suffices to show that 
  \begin{equation}
    \label{eq:1}
\{ x\in I: f_P(x) = i\}=\{ x\in J: f_P(x) = i\}.
  \end{equation}

Fix $1\le i\le n$, and let 
\[
r = |\{ x\in I: f_P(x) = i\}|  = |\{ x\in J: f_P(x) = i\}|.
\]
	By Lemma~\ref{lem:chain}, for each $i \in [n]$, $\{ x\in P: f_P(x) = i\}$ is a chain. 
Accordingly, we can write
\[
C:=\{x\in P: f_P(x) = i\} = \{c_1<_P c_2<_P\dots<_{P} c_t\}.
\]
Since $I$ is an ideal, $I\cap C$ is also an ideal of $C$.
Because $C$ is a chain this means 
\[
I\cap C = \{c_1<_P c_2<_P\dots<_{P} c_r\}.
\]
By the same argument we also have
\[
J\cap C = \{c_1<_P c_2<_P\dots<_{P} c_r\}.
\]
Therefore $I\cap C = J\cap C$, which is \eqref{eq:1}. This completes the proof.
\end{proof}

We are now ready to prove Theorem~\ref{thm_main_injection}.

\begin{proof}[Proof of Theorem~\ref{thm_main_injection}]
  
  Let $P,Q \in \P(w_0^{(n+1)})$ such that $\ind_\ad(P) = \ind_\ad(Q)$ for all
  $\ad \in \{\A,\D\}^{n-1}$. We will prove $P\sim Q$ by induction on $n$.

  Since $P,Q \in \P(w_0^{(n+1)})$, there are reduced words
  $\mathbf{i},\mathbf{j}\in \Rn{n+1}$ with $P\sim \Pi$ and $Q\sim \Pj$. If $n =
  1$, then $\Rn{n+1}$ has only one element $(1)$. Thus $\mathbf{i}=\mathbf{j}$,
  and $P\sim\Pi=\Pj\sim Q$. If $n = 2$, there are two reduced words $(1,2,1)$
  and $(2,1,2)$ in $\Rn{n+1}$. Since $\ind_{\A}(1,2,1) = 1$ and
  $\ind_{\A}(2,1,2) = 0$, if $\ind_\ad(\Pi) = \ind_\ad(\Pj)$ for all
  $\ad \in \{\A,\D\}^{n-1}$, we must have $\mathbf{i}=\mathbf{j}$. Therefore we also
  have $P\sim\Pi=\Pj\sim Q$.

	Now let $n > 2$ and suppose that the statement holds for $n-1$. Since
  $\ind_\ad(P) = \ind_\ad(Q)$ for all $\ad \in \{\A,\D\}^{n-1}$, by the
  definition of $\delta$-indices, we have
  \begin{align*}
	\ind_{\delta}(C_\D(P)) &= \ind_{\delta}(C_\D(Q)) \quad \text{ for all }\ad \in \{\A,\D\}^{n-2},\\
	\ind_{\delta}(C_\A(P)) &= \ind_{\delta}(C_\A(Q)) \quad \text{ for all }\ad \in \{\A,\D\}^{n-2}.
  \end{align*}
  Thus, by the induction hypothesis, we have $C_\D(P)\sim C_\D(Q)$ and $C_\A(P)\sim C_\A(Q)$.

  By Proposition~\ref{prop:EC}, 
  \begin{align*}
P &\sim E_\D(C_\D(P), I_\D(P)),\\
Q &\sim E_\D(C_\D(Q), I_\D(Q)).
  \end{align*}
  Since $C_\D(P)\sim C_\D(Q)$, in order to show $P\sim Q$, it suffices to show
  that the word poset isomorphism $C_\D(P)\sim C_\D(Q)$ induces $I_\D(P)\sim
I_\D(Q)$. By Lemma~\ref{lemma_order_ideal_n_of_elements}, in order to show
$I_\D(P)\sim I_\D(Q)$, it suffices to show the following claim: for all $i\in
[n]$,
\[
|\{ x\in I_\D(P): f_{C_\D(P)}(x)=i\}|=|\{ x\in I_\D(Q): f_{C_\D(Q)}(x)=i\}|.
\]
Since $f_{R}(x)=f_{C_\D(R)}(x)$ for all $x\in I_\D(R)$, where $R$ is $P$ or $Q$,
the claim can be rewritten as
\begin{equation}
  \label{eq:claim}
|\{ x\in I_\D(P): f_{P}(x)=i\}|=|\{ x\in I_\D(Q): f_{Q}(x)=i\}|.
\end{equation}

Let $R$ be either $P$ or $Q$. By Proposition~\ref{prop_index}, $\D(R)\cap\A(R)$
has a unique element, say $z$. Suppose $f_{R}(z)=s$. For $i\in [n]$, define
\begin{align*}
  a_i(R) &= |\{x\in I_\D(R)\setminus I_\A(R): f_R(x)=i\}|,\\
  b_i(R) &= |\{x\in I_\D(R)\cap I_\A(R): f_R(x)=i\}|,\\
  c_i(R) &= |\{x\in I_\A(R)\setminus I_\D(R): f_R(x)=i\}|.
\end{align*}
See Figure~\ref{fig:abc}.

\begin{figure}
  \centering
  	\begin{tikzpicture}[auto,node distance=0.1cm and 0.5cm, inner sep=2pt]
  \tikzstyle{A}=[draw, circle,  draw=red]
  \tikzstyle{D}=[draw, circle,  fill=blue!20]
  \tikzstyle{n} = [draw, circle]
  
  \node[n] (1) {}; 
  \node[n] (2) [above = of 1] {};
  \node[n] (3) [above = of 2] {};
  \node[D] (4) [above = of 3] {};
  \node[D] (5) [below right = of 4] {};
  \node[D] (6) [below right = of 5] {};
  \node[D] (7) [below right = of 6] {};
  \node[D] (8) [above left = of 4] {};
  \node[D, draw=red] (9) [above left = of 8] {\scriptsize $z$};
  \node[A] (10) [below left = of 9] {};
  \node[A] (11) [below left = of 10] {};
  \node[A] (12) [below left = of 11] {};
  \node[A] (13) [below left = of 12] {};
  \node[n] (14) [above = of 11] {}; 
  \node[n] (15) [above = of 14] {};
  \node[n] (16) [above = of 15] {};
  \node[n] (17) [above = of 16] {};
  \node[n] (18) [above = of 17] {};
  \node[D] (19) [above = of 18] {};
  \node[D] (20) [below right = of 19]{};
  \node[D] (21) [above left = of 19] {};
  \node[D] (22) [above left = of 21] {};
  \node[n] (23) [above = of 5] {};
  \node[n] (24) [above = of 23] {};
  \node[n] (25) [above = of 24] {};
  \node[n] (26) [above = of 25] {};
  \node[n] (27) [above = of 26] {};
  \node[n] (28) [above = of 27]{} ;
  \node[A] (29) [above = of 28] {};
  \node[A] (30) [below left = of 29] {};
  \node[A] (31) [above right = of 9] {};
  \node[A] (32) [above right = of 29] {};
  \node[A] (33) [above right = of 32] {};

\begin{scope}[on background layer]    
	\draw[rounded corners = 1mm, fill = teal!20!white, very thick, draw=none]
($(33.south)+(0.1,0)$)--($(30.south)+(0,-0.05)$)--($(31.south)+(0.05,-0.05)$)--($(9.south)+(0,-0.05)$)--($(7.south)+(0,-0.05)$)--($(7.south east)+(0.05,-0.05)$) --cycle node[midway, xshift = 1in] {\small \textcolor{teal}{$I_A(R) \setminus I_D(R)$}};

	\draw[rounded corners = 1mm, fill=orange!20!white, very thick, draw = none]
($(7.south)+(0,-0.3)$)--($(7.south)+(0,-0.05)$)--($(9.south)+(0,-0.05)$)--($(13.south)+(0,-0.05)$)--($(13.south)+(0,-0.55)$)--cycle node[midway, xshift=-0.5in, yshift=-0.5cm] {\small \textcolor{orange!50!red}{$I_D(R) \cap I_A(R)$}};

	\draw[rounded corners = 1mm, fill=purple!20!white, very thick, draw=none]
($(13.south west)+(-0.05,-0.05)$)--($(13.south)+(0,-0.05)$)--($(9.south)+(-0.1,-0.08)$)--($(20.south)+(0,-0.05)$)--($(22.south)+(0,-0.05)$)--($(22.south west)+(-0.05,-0.05)$)--cycle node[midway, xshift=-1in] {\small \textcolor{purple}{$I_D(R) \setminus I_A(R)$}};
\end{scope}

\draw[very thick, red] (13)--(12)--(11)--(10)--(9)--(31)--(30)--(29)--(32)--(33) node [right = 0.2cm] {\small $A(R)$};
\draw[very thick, dashed, blue] (22)--(21)--(19)--(20)--(9)--(8)--(4)--(5)--(6)--(7) node[right=0.2cm] {\small $D(R)$};

\draw[decorate,decoration={brace,amplitude=3pt}] 
($(1.west) + (-0.1,0)$)--($(3.west)+(-0.1,0)$) node[midway, xshift=-0.03in] {\small $b_j(R)$};
\draw[decorate,decoration={brace,amplitude=3pt}] 
($(28.east) + (0.1,0)$)--($(5.east)+(0.1,0)$) node[midway, xshift = 0.03in] {\small $c_k(R)$};
\draw[decorate,decoration={brace,amplitude=3pt}] 
($(11.west)+(-0.1,0)$)--($(18.west)+(-0.1,0)$) node[midway, xshift=-0.03in] {\small $a_i(R)$};

\draw[dotted,gray] (13)--(22)--($(22)+(0,0.9)$) node[above] {\small{$1$}};
\draw[dotted, gray] (11)--(14)--(15)--(16)--(17)--(18)--(19)--($(19)+(0,1.4)$) node[above] {\small{$i$}};
\draw[dotted, gray] (9)--($(9)+(0,2.8)$) node[above] {\small{$s$}};
\draw[dotted, gray] (1)--(2)--(3)--(4)--(30)--($(30)+(0,1.6)$) node[above] {\small{$j$}};
\draw[dotted, gray] (5)--(23)--(24)--(25)--(26)--(27)--(28)--(29)--($(29)+(0,1.4)$) node[above] {\small{$k$}};
\draw[dotted, gray] (7)--(33)--($(33)+(0,0.9)$) node[above] {\small{$n$}};

  \end{tikzpicture}
  \caption{An illustration of $I_\D(R)\setminus
    I_\A(R)$, $I_\D(R)\cap I_\A(R)$, $I_\A(R)\setminus I_\D(R)$, $a_i(R)$,
    $b_k(R)$, and $c_j(R)$.}
  \label{fig:abc}
\end{figure}
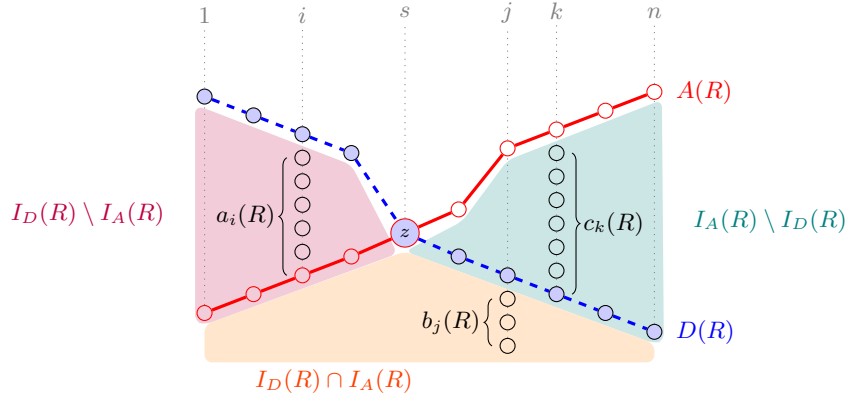

By definition,
\begin{equation}\label{eq:R}
  |\{x\in I_\D(R): f_R(x)=i\}| =
  \begin{cases}
    a_i(R)+b_i(R) & \mbox{if $i<s$},\\
    b_i(R) & \mbox{if $i\geq s$}.
  \end{cases}
\end{equation}
By Lemma~\ref{lemma_An_An-1_in_PD}, $\D(R)\setminus\{z\}$ is the descending
chain of $C_\A(R)$, which is obtained from $R$ by removing $\A(R)$ and shifting
the elements below $\A(R)$ to the left by one column. This shows that
\begin{equation}\label{eq:RA}
|\{x\in I_\D(C_\A(R)): f_{C_\A(R)}(x)=i\}| = 
  \begin{cases}
    a_i(R)-1+b_{i+1}(R) & \mbox{if $i<s$},\\
    b_{i+1}(R) & \mbox{if $i\ge s$}.
  \end{cases}
\end{equation}
Similarly, we have
\begin{equation}\label{eq:RD}
|\{x\in I_\A(C_\D(R)): f_{C_\D(R)}(x)=i\}| = 
  \begin{cases}
    b_{i}(R) & \mbox{if $i<s$},\\
    c_{i+1}(R)-1+b_{i+1}(R) & \mbox{if $i\ge s$}.
  \end{cases}
\end{equation}

Since $C_\A(P)\sim C_\A(Q)$ (respectively, $C_\D(P)\sim C_\D(Q)$), the left hand
side of \eqref{eq:RA} (respectively, \eqref{eq:RD}) is the same for both cases
$R=P$ and $R=Q$. Comparing the right hand sides of \eqref{eq:RA} and
\eqref{eq:RD} for the cases $R=P$ and $R=Q$, we obtain $b_i(P)=b_i(Q)$ for all
$1\le i\le n$, $a_i(P)=a_i(Q)$ for all $1\le i\le s-1$, and $c_i(P)=c_i(Q)$ for
all $s+1\le i\le n$. By \eqref{eq:R}, this implies the claim \eqref{eq:claim}
and the proof is completed.
\end{proof}

	We note that 	B\'{e}dard~\cite{Bedard99} studied the combinatorics of commutation classes for Weyl groups of any Lie types by introducing a \textit{level function} on a certain subset of positive roots of the corresponding root system. Indeed, for each reduced word a level function is defined, and this function distinguishes commutation classes, i.e., $\mathbf i \sim \mathbf j$ if and only if the corresponding level functions are the same. 

\begin{Question}
  Recall that the indices have been defined for the reduced words of the longest
  element in $\mathfrak{S}_{n+1}$, which is the Weyl group of Lie type $A$. The
  level functions introduced by B\'{e}dard~\cite{Bedard99} and string polytopes
  are defined for any Lie type. In this regard, we may ask whether one can
  generalize the definitions of indices to other Lie types to provide more
  fruitful understanding of the combinatorics of string polytopes.
\end{Question}
\begin{remark}
  We have seen that the indices of reduced words are used to classify the string
  polytopes combinatorially equivalent to a Gelfand--Cetlin polytope. Recently, the
  combinatorics of string polytopes associated with reduced words of
  \textit{small indices} has been studied in~\cite{CKLP2}. A reduced word
  $\mathbf i \in \Rn{n+1}$ has small indices if $\ind_{\ad}(\mathbf i) =
  (0,\dots,0,k )$ for some $\ad \in \{\A,\D\}^{n-1}$ and $k \leq
  \kappa(\ad_{n-1},\ad_n)$. Here, $\kappa(\ad_{n-1},\ad_n) = 2$ if $\ad_n =
  \ad_{n-1}$; and $\kappa(\ad_{n-1},\ad_n) = n-1$ otherwise. In~\cite{CKLP2}, Cho
  et al. found the number of codimension one faces and the description of the
  vertices for the string polytopes associated with reduced words having small
  indices. These examples show that the notion of indices may have a potential
  role to study the combinatorics of string polytopes.
\end{remark}

%

\providecommand{\bysame}{\leavevmode\hbox to3em{\hrulefill}\thinspace}
\providecommand{\MR}{\relax\ifhmode\unskip\space\fi MR }
\providecommand{\MRhref}[2]{%
	\href{http://www.ams.org/mathscinet-getitem?mr=#1}{#2}
}
\providecommand{\href}[2]{#2}

\end{document}